\documentclass[11pt,twoside, leqno]{article}

\usepackage{amsthm}
\usepackage{amssymb}
\usepackage{amsmath}
\usepackage{mathrsfs}
\usepackage{txfonts}
\usepackage{graphics}
\usepackage[Symbol]{upgreek}

\allowdisplaybreaks 
\def\rr{{\mathbb R}}

\def\cn{{\mathbb N}}

\def\fz{\infty}

\def\supp{{\mathop\mathrm{\,supp\,}}}

\def\dz{\delta}

\def\ez{\epsilon}

\def\gz{{\gamma}}

\def\vz{\varphi}

\def\sz{\sigma}

\def\hs{\hspace{0.3cm}}

\def\r{\right}
\def\lf{\left}
\def\lfr{\lfloor}
\def\rf{\rfloor}
\def\la{\langle}
\def\ra{\rangle}

\newtheorem{thm}{Theorem}[section]
\newtheorem{lem}{Lemma}[section]
\newtheorem{prop}{Proposition}[section]
\newtheorem{rem}{Remark}[section]
\newtheorem{cor}{Corollary}[section]
\newtheorem{defn}{Definition}[section]

\numberwithin{equation}{section}

\begin{document}
\arraycolsep=1pt
\author{Renjin Jiang} \arraycolsep=1pt
\arraycolsep=1pt

\title{\bf\Large The Li-Yau Inequality and Heat Kernels \\
on Metric Measure Spaces
\footnotetext{\hspace{-0.35cm}
2010 \emph{Mathematics Subject Classification}. Primary 53C23; Secondary 58J35, 31E05.
\endgraf
{\it Key words and phrases}. Li-Yau inequality, Harnack inequality, heat kernel, metric measure space.
\endgraf
The author was partially supported by NSFC (No. 11301029),
the Fundamental Research Funds for Central Universities of China (No. 2013YB60) and
the Project-sponsored by SRF for ROCS, SEM.
}}
\author{Renjin Jiang}
\date{}

\maketitle

\begin{center}
\begin{minipage}{11cm}\small
{\noindent{\bf Abstract} Let $(X,d,\mu)$ be a $RCD^\ast(K, N)$ space with $K\in \rr$ and $N\in [1,\infty)$.
Suppose that $(X,d)$ is connected, complete and separable, and $\supp \mu=X$.
We prove that the Li-Yau inequality for the heat flow holds true on $(X,d,\mu)$ when $K\ge 0$. A Baudoin-Garofalo inequality and
Harnack inequalities for the heat flow are established on $(X,d,\mu)$ for general $K\in \rr$.
Large time behaviors of heat kernels are also studied.}\end{minipage}

\vspace{1cm}

\begin{minipage}{11cm}\small
{\noindent{\bf R\'esum\'e} Soit $(X,d,\mu)$ un espace $RCD^\ast(K, N)$ avec $K\in \rr$ et $N\in [1,\infty)$. On suppose que $(X,d)$ est connexe, complet et s\'eparable et $\supp \mu=X$.
Nous d\'emontrons que l'in\'egalit\'e de Li-Yau pour le flot de la chaleur sur $(X,d,\mu)$ est satisfaite lorsque $K\ge 0$. De plus, nous \'etablissons une in\'egalit\'e de type Baudoin-Garofalo ainsi que des in\'egalit\'es de Harnack pour ce flot dans la situation plus g\'en\'erale o\`u $K\in \rr$. Le comportement en temps long des noyaux de la chaleur est aussi \'etudi\'e.}\end{minipage}

\end{center}

\vspace{0.2cm}

\section{Introduction}
\hskip\parindent Non-smooth calculus on metric measure spaces as a generalization from the classical smooth settings has attracted intensive interest in the last several decades; see, for instance, \cite{ags1,ags2,ags3,ams13, ch,eks13,gi12,he07,hek,lv09,raj1,sh,stm4,stm5} and references therein. In this article, we deal with the Li-Yau inequality and
Harnack inequalities for the heat flow.

It is well known that, in Riemannian geometry, Ricci curvature bounded from below is essential
for many analytic and geometric properties of Riemannian manifolds. However, Riemannian manifolds with Ricci curvature bounded from below are not stable under Gromov-Hausdorff convergence; we refer the reader to Cheeger and Colding \cite{cc1,cc2,cc3}
for comprehensive studies on the Gromov-Hausdorff limit space of manifolds with Ricci curvature bounded below.

On complete metric spaces, using optimal transportation, Lott-Villani \cite{lv09} and Sturm \cite{stm4,stm5}
introduced the Ricci curvature condition $CD(K, N)$ for $K\in \rr$ and $N\in [1,\infty]$ (when $N<\infty$, only $CD(0,N)$ condition was introduced in \cite{lv09}), which is stable with respect to the measured Gromov-Hausdorff convergence. Precisely, a complete metric space satisfying $CD(K, N)$ means the space having Ricci curvature bounded below by $K\in\rr$ and dimension bounded above by $N\in [1,\infty]$. When backing into the Riemannian setting, the $CD(K, N)$ condition is equivalent
to the requirement that the space has Ricci curvature bounded from below by $K$.
In order to overcome the possible lack of local-to-global properties under $CD(K,N)$ conditions for finite $N$,
Bacher and Sturm \cite{bas10} introduced the so-called reduced curvature-dimension condition $CD^\ast(K,N)$ for
$K\in \rr$ and $N\in [1,\infty)$, which is equivalent to the local version of $CD(K,N)$ under the non-branching condition.

On the other hand,  on a Riemannian manifold $M$, the Bakry-\'Emery condition
$$\Gamma H_tf\le e^{-2Kt}H_t\Gamma f, \ \ \forall \, f\in C^\infty_c(M),$$
is a generalization of the notion of the Ricci curvature and is a powerful tool in geometric analysis, where $\Gamma$ is the carr\'e du champ operator and $H_t$ is the corresponding heat flow; see \cite{bak1,be2}.
On metric measure spaces,  it was shown in \cite{krs,ji14} that a locally doubling measure,
a local weak $L^2$-Poincar\'e inequality and a Bakry-\'Emery type inequality are sufficient to
guarantee the Lipschitz continuity of Cheeger-harmonic functions.

Since the $CD(K,N)$, $CD^\ast(K,N)$ conditions include Finsler geometry, it is not known if
the Bakry-\'Emery type conditions hold under them.
Recently, in a series of seminal works, Ambrosio, Gigli and Savar\'e  \cite{ags1,ags2,ags3},
and Ambrosio, Gigli, Mondino and  Rajala \cite{agmr} developed the Riemannian curvature dimension
$RCD(K,\infty)$ by further requiring the $CD(K,\infty)$ spaces being infinitesimally Hilbertian (see Section 2 for precise definition),
and identified the $RCD(K,\infty)$ condition with the Bakry-\'Emery condition $BE(K,\infty)$.

The finite dimensional Riemannian curvature-dimension $RCD^\ast(K,N)$ condition was later introduced by
Erbar, Kuwada and Sturm \cite{eks13}, and  Ambrosio, Mondino and Savar\'e \cite{ams13-1}.
In particular, the $RCD^\ast(K,N)$ condition implies the Bochner inequality $BE(K,N)$ and the Bakry-Ledoux pointwise gradient estimate $BL(K,N)$,  moreover, these conditions
are equivalent under some mild regularity assumptions; see \cite{eks13,ams13-1}.
For local-to-global property of $RCD^\ast(K,N)$ spaces we refer to \cite{ams13}.
Very recently, Gigli \cite{gi13} has obtained the splitting theorem for $RCD^\ast(0,N)$ spaces.
The gradient estimates of harmonic functions on $RCD^\ast(K,N)$ spaces are established in \cite{ji14,hkx13}.
Our main purpose of the article is to establish the Li-Yau inequality and the Harnack inequality for the heat flow.

The gradient estimates and Harnack inequalities of solutions to the heat equation, obtained by Li and Yau \cite{ly86}, are fundamental
tools in geometric analysis. On a Riemannian manifold $M$ with non-negative Ricci
curvature, the Li-Yau inequality states that any positive solution $u$ to the heat equation satisfies
$$|\nabla \log u|^2-\frac{\partial}{\partial t}\log u\le \frac {n}{2t},$$
which in turn implies the Harnack inequality for $u$: for any $0<s<t<\infty$ and $x,y\in M$,
$$u(x,s)\le u(y,t)\exp\lf\{\frac{d(x,y)^2}{4(t-s)}\r\}\lf(\frac t s\r)^{N/2}.$$

Recently, there have been some efforts to generalize the Li-Yau inequality and Harnack inequalities to the metric measure settings.
Qian, Zhang and Zhu \cite{qzz13} established the above Li-Yau inequality and Harnack inequalities on
compact Alexandrov spaces with non-negative Ricci curvature, where the Ricci curvature on Alexandrov spaces
was introduced by Zhang and Zhu in \cite{zz10}. Notice that it was shown in \cite{zz10} that such spaces
are $CD(0,N)$, and hence they  satisfy the $RCD^\ast(0,N)$ condition since the heat flow
on Alexandrov spaces is linear.

On $RCD^\ast(K, \infty)$ spaces, a dimension free Harnack inequality for the heat semigroup was obtained by Li \cite{li13}.
On a $RCD^\ast(K, N)$ space $(X,d,\mu)$ with $N<\infty$ and $\mu$ being a probability measure,
Garofalo and Mondino \cite{gam14} established the Li-Yau inequality for $K=0$, and
Harnack inequalities for general $K\in \rr$. Although it was required $\mu(X)=1$,
it is easy to see that their results work for general cases as soon as $\mu(X)<\infty$.
However, the case $\mu(X)=\infty$ remains open so far.

Our main result below completes the case $\mu(X)=\infty$. Throughout this article, we assume that
$(X,d)$ is connected, complete and separable, and $\mu$ is a locally finite, $\sigma$-finite Borel measure with
$\supp \mu=X$.

\begin{thm}[Li-Yau Inequality]\label{lye}
Let $(X,d,\mu)$ be a $RCD^\ast(0,N)$ space with $N\in [1,\infty)$.
Assume that $u(x,t)$ is a solution to the heat equation on $X\times [0,\infty)$ with the initial value
$u(x,0)=f(x)$, where $0\le f\in \cup_{1\le q<\infty}L^q(X)$.
Then it holds for each $T>0$ that
\begin{equation}
|\nabla\log u(x,T)|^2-\frac{\partial}{\partial t}\log u(x,T)\le \frac N {2T},\ \ \ \mu-a.e.\, x\in X.
\end{equation}
\end{thm}
Above and in what follows,  $|\nabla u|$ denotes the minimal weak upper gradient of $u$; see Section 2 below.
Moreover, we always assume that the initial value $f$  is not identically zero in the $\mu$-a.e. sense.

Garofalo and Mondino in \cite{gam14} obtained Theorem \ref{lye} for the case $\mu(X)=1$ (whose proof works also
for general cases $\mu(X)<\infty$), by a non-trivial adaption of a purely analytical approach to the
Li-Yau inequality from Baudoin and Garofalo \cite{bg11} in the Riemannian manifold. Notice that for
$(X,d,\mu)$ being a $RCD^\ast(0,N)$ space, if $\mu(X)<\infty$, then $(X,d)$ has to be bounded. Indeed, since
$(X,d)$ is a geodesic space (\cite[Remark 3.18]{eks13}) and the measure $\mu$ is doubling, there exists $C>1$ such that,
for each $x\in X$ and $r\in (0,\mathrm{diam}(X)/2)$, it holds
$$\mu(B(x,2r))\ge C\mu(B(x,r)),$$
which is the so-called reverse doubling condition; see \cite[Proposition 2.1]{yz11}. This implies that, if $\mu(X)<\infty$,
then $(X,d)$ has to be bounded.

Our arguments will be based on the arguments given in \cite{gam14} with some necessary and non-trivial
generalizations. In particular, following \cite{gam14}, we shall use the following functional
\begin{equation*}
\Phi(t):=H_{t}\lf(H_{T-t}f_\delta|\nabla \log H_{T-t}f_\delta|^2\r),
\end{equation*}
where $H_t$ is the heat flow, $f_\delta=f+\delta$, $0\le f\in L^1(X)\cap L^\infty(X)$ and $\delta>0$, and then use the integral
$\int_X\Phi(t)\vz\,d\mu$ as our main object.
A key step of the proof is to use
the Bochner inequality (see \cite{eks13} or Theorem \ref{Bochner}) for the function
$\log H_{T-t}f_\delta$. Since $\log H_{T-t}f_\delta$ may not be in the Sobolev space
$W^{1,2}(X)$ when $\mu(X)=\infty$, we need a generalized  Bochner inequality from \cite[Corollary 4.3]{ams13}; see also
Theorem \ref{gbi} below. Notice that  \cite[Corollary 4.3]{ams13} actually provides a stronger result than our Theorem \ref{gbi}, whose proof will be kept for completeness.

Another critical point is that, unlike the compact case, the heat kernel $p_t$ may not be a bounded function on $X\times X$, and $H_{t}$ may not be bounded from $L^1(X)$ to $L^\infty(X)$.
Indeed, Proposition \ref{q-infty-bound} below shows that $H_{t}$ is bounded from $L^q(X)$ to $L^\infty(X)$ for some (all) $q\in [1,\infty)$ if and only if $\inf_{x\in X}\mu(B(x,1))>0$. Notice that there exists a complete Riemannian manifold $M$,
with non-negative Ricci curvature, that satisfies $\inf_{x\in M}\mu(B(x,1))=0$; see Croke and Karcher \cite{ck88}.
Therefore, we do not know the absolute continuity of the map  $t\mapsto\int_X\Phi(t)\vz\,d\mu $ for general
$f\in \cup_{1\le q<\infty}L^q(X)$ and $\vz \in L^1(X)\cap L^\infty(X)$. Here we will combine some methods from harmonic analysis.
Precisely, we shall establish the boundedness of $|\nabla H_tf|$ on $L^p(X)$ for $p\in [1,\infty]$ by
using a rough gradient estimate of the heat kernel from \cite{jkyz14}; see Section 3 below.
Then we first prove Theorem \ref{lye} for $f\in L^1(X)\cap L^\infty(X)$, and finally a limiting argument gives the desired result.

A direct corollary is the following Li-Yau inequality for the heat kernel.
\begin{cor}\label{ly-heatkernel}
Let $(X,d,\mu)$ be a $RCD^\ast(0,N)$ space with $N\in [1,\infty)$.
Let $p$ be the heat kernel. Then, for $\mu$-a.e. $x,y\in X$ and each $t>0$, it holds
\begin{equation}
|\nabla_x \log p_t(x,y)|^2-\frac{\partial}{\partial t}\log p_t(x,y)\le \frac N {2t}.
\end{equation}
\end{cor}

By following the proofs from \cite[Theorem 1.3]{gam14} and using the tools established in proving Theorem \ref{lye},
we can obtain the following Baudoin-Garofalo inequality for the heat flow; see \cite{bg11,gam14}.
\begin{thm}[Baudoin-Garofalo Inequality]\label{BGI}
Let $(X,d,\mu)$ be a $RCD^\ast(K,N)$ space with $K\in \rr$ and $N\in [1,\infty)$.
Then, for every $0\le f\in \cup_{1\le q<\infty}L^q(X)$, it holds for each $T>0$ that
$$|\nabla \log H_Tf|^2\le e^{-2KT/3}\frac{\Delta H_Tf}{H_Tf}+\frac{NK}3\frac{e^{-4KT/3}}{1-e^{-2KT/3}},\ \ \mu-a.e.,$$
where $\frac{NK}3\frac{e^{-4KT/3}}{1-e^{-2KT/3}}$ is understood as $\frac N{2T}$ when $K=0$.
\end{thm}

The Baudoin-Garofalo inequality further implies the following Harnack inequality for the heat flow.
Notice that  the proof of \cite[Theorem 1.4]{gam14} works directly for the following theorem, whose proof will be omitted.
\begin{thm}[Harnack Inequality]\label{harnack}
Let $(X,d,\mu)$ be a $RCD^\ast(K,N)$ space with $K\in \rr$ and $N\in [1,\infty)$.
Then for each $0\le f\in \cup_{1\le q<\infty}L^q(X)$, all $0<s<t<\infty$ and $x,y\in X$, it holds that

{\rm (i)} if $K>0$,
$$H_sf(x)\le H_tf(y)\exp\lf\{\frac{d(x,y)^2}{4(t-s)e^{2Ks/3}}\r\}\lf(\frac{1-e^{2Kt/3}}{1-e^{2Ks/3}}\r)^{N/2};$$

{\rm (ii)} if $K=0$,
$$H_sf(x)\le  H_tf(y)\exp\lf\{\frac{d(x,y)^2}{4(t-s)}\r\}\lf(\frac t s\r)^{N/2};$$

{\rm (iii)} if $K<0$,
$$H_sf(x)\le H_tf(y)\exp\lf\{\frac{d(x,y)^2}{4(t-s)e^{2Kt/3}}\r\}\lf(\frac{1-e^{2Kt/3}}{1-e^{2Ks/3}}\r)^{N/2}.$$
\end{thm}

As an application of the Harnack inequality, we shall prove the following large time behavior of heat kernels on
$RCD^\ast(0,N)$ spaces with maximal volume growth; see Li \cite{li86}.
\begin{thm}\label{large-time}
Let $(X,d,\mu)$ be a $RCD^\ast(0,N)$ space with $N\in [1,\infty)$. Let $x_0\in X$.
If there exists $\theta\in (0,\infty)$ such that $\liminf_{R\to\infty}\frac{\mu(B(x_0,R))}{R^N}=\theta$,
then there exists a constant $C(\theta)\in (0,\infty)$ such that, for any
$x,y\in X$, it holds that
$$\lim_{t\to\infty} \mu(B(x_0,\sqrt t))p_t(x,y)=C(\theta).$$
\end{thm}

The paper is organized as follows. In Section 2, we give some basic notation
and notions for  Sobolev spaces, differential structures,  curvature-dimension conditions and
heat kernels. Section 3 is devoted to establishing a rough gradient estimate for
the heat kernels and the mapping properties of $|\nabla H_t|$. In Section 4,
we deal with the generalized Bochner inequality.
Theorem \ref{lye} and Corollary \ref{ly-heatkernel} will be proved in Section 5,
the Baudoin-Garofalo inequality (Theorem \ref{BGI}) and the Harnack inequality (Theorem \ref{harnack})
will be proved in Section 6. In the final section, we will apply the Harnack inequality to
study the large time behavior of heat kernels, and prove Theorem \ref{large-time} there.

Finally, we make some conventions on notation. Throughout the paper, we denote
by $C,c$ positive constants which are independent of the main
parameters, but which may vary from line to line. The symbol
$B(x,R)$ denotes an open ball
with center $x$ and radius $R$, and $CB(x,R)=B(x,CR).$ For any real values $a$ and $b$,
let $a\wedge b:=\min\{a,\,b\}$. The space $LIP(X)$ denotes the set of all Lipschitz functions
on $X$.

\section{Preliminaries}
\hskip\parindent In this section, we recall some basic notions and several auxiliary results.

\subsection{Sobolev spaces on metric measure spaces}
\hskip\parindent Let $C([0, 1],X)$ be the space of continuous curves on $[0, 1]$ with values in $X$, which we
endow with the sup norm. For $t\in [0, 1]$, the map $e_t :C([0, 1],X) \to X$ is the evaluation at
time $t$ defined by
$$e_t(\gz):=\gz_t.$$
Given a non-trivial closed interval $I\subset \rr$, a curve $\gz: I \to X$ is in the absolutely continuous
class $AC^q([0,1],X)$ for some $q\in [1,\infty]$, if there exists $f\in L^q(I)$ such that, for all $s,t\in I$ and $s<t$, it holds
$$d(\gz_t,\gz_s)\le\int_s^tf(r)\,dr.$$

\begin{defn}[Test Plan]
Let $\uppi\in {\mathcal{P}}(C([0, 1],X))$.
We say that $\uppi$ is a test plan if there exists $C>0$ such that
$$(e_t)_\sharp{ \uppi}\le C\mu, \ \forall\,t\in [0,1],$$
and
$$\int \int_0^1|\dot{\gamma}_t|^2\,dt\,d{\mathcal \uppi}(\gamma)<\infty.$$
  \end{defn}

\begin{defn}[Sobolev Space] \label{sobolev} The Sobolev class $S^2(X)$ (resp. $S_{\mathrm{loc}}^2(X)$)
is the space of all Borel functions $f: X\to \rr$, for which there exists a non-negative function
$G\in L^2(X)$ (resp. $G\in L^2_{\mathrm{loc}}(X)$) such that, for each test plan $\uppi$, it holds
\begin{equation}\label{curve-sobolev}
\int |f(\gz_1)-f(\gz_0)|\,d{\mathcal \uppi}(\gamma)\le \int \int_0^1 G(\gz_t)|\dot{\gamma}_t|^2\,dt\,d{\mathcal \uppi}(\gamma).
\end{equation}
\end{defn}

By a compactness argument (see \cite{ch,sh,ags4}), for each  $f\in S^2(X)$ there exists a unique minimal $G$
in the $\mu$-a.e. sense such that \eqref{curve-sobolev} holds. We then denote the minimal $G$ by $|\nabla  f|$
and call it the minimal weak upper gradient following \cite{ags4}.

We then define the in-homogeneous Sobolev space $W^{1,2}(X)$ as  $S^2(X)\cap L^2(X)$ equipped with the norm
$$\|f\|_{W^{1,2}(X)}:=\lf(\|f\|_{L^2}^2+\| |\nabla  f|\|_{L^2(X)}^2\r)^{1/2}.$$

\begin{defn}[Local Sobolev Space]  Let $\Omega\subset X$ be an open set. A Borel function $f: \Omega\to \rr$ belongs to
$S^2_{\mathrm{loc}}(\Omega)$, provided, for any Lipschitz function $\chi: X \to\rr$ with $\supp(\chi)\subset \Omega$,  it holds
$f\chi\in S^2_{\mathrm{loc}}(X)$. In this case, the function $|\nabla f| : \Omega\to[0,\infty]$ is $\mu$-a.e. defined by
$$|\nabla f| := |\nabla (\chi f)|,\ \  \mu-a.e. \,\mathrm{on} \ \{\chi=1\},$$
 for any $\chi$ as above. The space $S^2(\Omega)$ is the collection of such $f$ with $|\nabla f|\in L^2(\Omega)$.
\end{defn}

The local Sobolev space $W^{1,2}_{\mathrm{loc}}(\Omega)$ for an open set $\Omega$,  and the Sobolev space with compact support
$W^{1,2}_{c}(X)$ can be defined in an obvious manner.  Notice that the Sobolev space $W^{1,2}(X)$ coincides with the Sobolev spaces based on upper gradients introduced by Cheeger \cite{ch} and Shanmugalingam \cite{sh}; see Ambrosio, Gigli and Savar\'e \cite{ags4}.

\subsection{Differential structure and the Laplacian}
\hskip\parindent  The following terminologies and results are mainly taken from \cite{ags3,gi12}.
\begin{defn}[Infinitesimally Hilbertian Space] Let $(X, d,\mu)$ be a proper metric measure
space. We say that it is infinitesimally Hilbertian, provided $W^{1,2}(X)$ is a Hilbert space.
\end{defn}
Notice that, from the definition, it follows that $(X, d,\mu)$ is infinitesimally Hilbertian if and only if,
for any $f,g\in S^2(X)$, it holds
$$\||\nabla (f+g)|\|_{L^2(X)}^2+\||\nabla (f-g)|\|_{L^2(X)}^2=2\lf(\||\nabla f|\|_{L^2(X)}^2+\||\nabla g|\|_{L^2(X)}^2\r).$$

\begin{defn} Let $(X, d,\mu)$ be an infinitesimally Hilbertian space, $\Omega\subset X$ an open set and $f, g\in S^2_{\mathrm{loc}}(\Omega)$.
The map $\la \nabla f, \nabla g\ra :\, \Omega \to \rr$ is $\mu$-a.e. defined as
$$\la \nabla f, \nabla g\ra:= \inf_{\ez>0} \frac{|\nabla (g+\ez f)|^2-|\nabla g|^2}{2\ez}$$
the infimum being intended as $\mu$-essential infimum.
\end{defn}

We shall sometimes write $\la \nabla f, \nabla g\ra$ as $\nabla f\cdot\nabla g$ for convenience.
We next summarize some basic properties of $\la \nabla f, \nabla g\ra$.
\begin{prop}
Let $(X, d,\mu)$ be an infinitesimally Hilbertian space and $\Omega\subset X$ an open set.
Then $W^{1,2}(\Omega)$ is a Hilbertian space, and  the following holds.

{\rm (i) Cauchy-Schwartz inequality:} For $f, g\in S^2_{\mathrm{loc}}(\Omega)$, it holds $\la  \nabla  f,  \nabla f\ra=|\nabla f|^2$
and $$|\la \nabla f, \nabla g\ra|\le |\nabla f||\nabla g|, \ \ \mu-a.e.\, \mathrm{on } \ \Omega.$$

{\rm (ii) Linearity:} For $f, g,h\in S^2_{\mathrm{loc}}(\Omega)$, and $\alpha,\beta\in\rr$, it holds
$$\la  \nabla (\alpha f+\beta h),  \nabla g\ra= \alpha\la \nabla f,  \nabla g\ra+\beta \la \nabla h,  \nabla g\ra, \ \ \mu-a.e. \,\mathrm{on } \ \Omega.$$

{\rm (iii) Chain rule:} For $f, g\in S^2_{\mathrm{loc}}(\Omega)$, and $\vz:\rr\to\rr$ Lipschitz, it holds
$$\la  \nabla (\vz\circ f),  \nabla g\ra= \vz'\circ f\la  \nabla f,  \nabla g\ra, \ \ \mu-a.e. \,\mathrm{on } \ \Omega.$$

{\rm (iv) Leibniz rule:} For $f, h\in S^2_{\mathrm{loc}}(\Omega)\cap L^\infty_{\mathrm{loc}}(\Omega)$ and $g\in S^2_{\mathrm{loc}}(\Omega)$, it holds
$$\la  \nabla (fh),  \nabla g\ra= h\la  \nabla f,  \nabla g\ra+f\la  \nabla h,  \nabla g\ra, \ \ \mu-a.e.\, \mathrm{on } \ \Omega.$$
\end{prop}

With the aid of the inner product, we can define the Laplacian operator as below. Notice that the Laplacian operator is linear
due to $(X, d,\mu)$ being infinitesimally Hilbertian.

\begin{defn}[Laplacian] \label{glap}
Let $(X, d,\mu)$ be an infinitesimally Hilbertian space.
Let $f\in W^{1,2}_{\mathrm{loc}}(X)$. We say that $f\in {\mathcal D}_{\mathrm{loc}}(\Delta)$, if there exists $h\in L^1_{\mathrm{loc}}(X)$ such that, for each $\psi\in W^{1,2}_{c}(X)$, it holds
\begin{equation*}
\int_X\la \nabla f,\nabla\psi\ra\,d\mu=-\int_X h\psi\,d\mu.
\end{equation*}
We will write $\Delta f=h$. If $f\in W^{1,2}(X)$ and $h\in L^2(X)$, we then say that $f\in {\mathcal D}(\Delta)$.
\end{defn}

From the Leibniz rule, it follows that if $f,g\in {\mathcal D}_{\mathrm{loc}}(\Delta)\cap L^\infty_{\mathrm{loc}}(X)$ (resp. $f,g\in {\mathcal D}(\Delta)\cap L^\infty(X)\cap LIP(X)$), then $fg\in {\mathcal D}_{\mathrm{loc}}(\Delta)$ (resp. $f,g\in {\mathcal D}(\Delta)$) satisfies
$\Delta (fg)=g\Delta f+f\Delta g+2\nabla f\cdot\nabla g$.

\subsection{Curvature-dimension conditions and consequences}
\hskip\parindent Let $(X,d,\mu)$ be  an infinitesimally Hilbertian space. Denote by $H_t$  the heat flow $e^{t\Delta}$ generated
from the Dirichlet forms $\int_X\la \nabla f,\nabla\psi\ra\,d\mu$. From $(X,d,\mu)$ being infinitesimally Hilbertian, it follows
that $H_t$ is linear.

%

We shall use the following definition for $RCD^\ast(K,N)$ spaces, which is slightly weaker than the original definition from \cite{eks13},
and is equivalent to the original definition under mild regularity assumptions; see \cite{eks13,ams13-1}.

\begin{defn}[$RCD^\ast(K,N)$ Space]\label{rcd}
Let $(X,d,\mu)$ be  an infinitesimally Hilbertian space. The space $(X,d,\mu)$ is called a $RCD^\ast(K,N)$ space
  for some $K\in \rr$ and $N\in [1,\infty)$, if, for all $f\in W^{1,2}(X)$ and each $t>0$, it holds that
\begin{equation}
|\nabla H_tf(x)|^2+\frac{4Kt^2}{N(e^{2Kt}-1)}|\Delta H_tf(x)|^2\le e^{-2Kt}H_t(|\nabla f|^2)(x)
\end{equation}
$\mu$-a.e. $x\in X$.
\end{defn}

An important tool is the following Bochner inequality; see \cite{eks13}.
\begin{thm}[Bochner Inequality]\label{Bochner}
Let $(X,d,\mu)$ be  a $RCD^\ast(K,N)$ space, where $K\in \rr$ and $N\in [1,\infty)$.
Then, for all $f\in \mathcal{D}(\Delta)$ with $\Delta f\in W^{1,2}(X)$ and all $g\in \mathcal{D}(\Delta)$ bounded and non-negative
with $\Delta g\in L^{\infty}(X)$, it holds
\begin{equation}
  \frac 12 \int_X \Delta g|\nabla f|^2\,d\mu- \int_X g\langle \nabla f,\nabla \Delta f\rangle\,d\mu\ge K\int_X g|\nabla f|^2\,d\mu +\frac{1}{N}\int_X g(\Delta f)^2\,d\mu.
\end{equation}
\end{thm}

For $(X,d,\mu)$ being  a $RCD^\ast(K,N)$ space, the measure $\mu$ is known to be  locally doubling (globally doubling,
if $K\ge 0$) according to \cite{eks13}.
\begin{lem}\label{volume-growth}
Let $(X,d,\mu)$ be a $RCD^\ast(K,N)$ space with $K\le 0$ and $N\in [1,\infty)$. Let $x\in X$ and $0<r<R<\infty$.

{\rm(i)} If $K=0$, then $$\frac{\mu(B(x,R))}{\mu(B(x,r))}\le \lf(\frac Rr\r)^N.$$

{\rm(ii)} If $K<0$,
then $$\frac{\mu(B(x,R))}{\mu(B(x,r))}\le \frac{\ell_{K,N}(R)}{\ell_{K,N}(r)},$$
where $\ell_{K,N}$, depending on $K,N$, is an increasing function on $(0,\infty)$, and $\ell_{K,N}(R)=O(e^{c_{K,N}R})$
as $R\to \infty$, for some constant $c_{K,N}$ depends on $K,N$.
\end{lem}

Rajala \cite{raj1,raj2} showed that a local weak $L^2$-Poincar\'e inequality holds on $RCD^\ast(K,N)$ spaces,
and a uniform $L^2$-Poincar\'e inequality holds if $K\ge 0$.
Hence, the results from Sturm \cite{st2,st3} imply that there exists $C:=C(N,K)$ ($C:=C(N)$, if $K\ge 0$) such that,
for each $t\le 1$ (resp. all $t>0$) and all $x,y\in X$, it holds that
\begin{eqnarray}\label{hk}
  &&C^{-1}\mu(B(x,{\sqrt t}))^{-1/2}\mu(B(y,{\sqrt t}))^{-1/2}\exp\lf\{-\frac{d(x,y)^2}{C_2t}\r\}\nonumber\\
  &&\hs \hs\hs\le p(t,x,y)\le
  C\mu(B(x,{\sqrt t}))^{-1/2}\mu(B(y,{\sqrt t}))^{-1/2}\exp\lf\{-\frac{d(x,y)^2}{C_1t}\r\}.
\end{eqnarray}

Furthermore, since we have assumed that $\supp \mu=X$ and $(X,d)$ is connected, by \cite[Remark 3.18]{eks13} we know that
the $RCD^\ast(K,N)$ space $(X,d,\mu)$ is a geodesic space.
Thus, for all $x,y\in X$, there is a curve $\gz$ connecting $x$ and $y$ and satisfying $length(\gz)=d(x,y)$.
This and the local doubling condition imply that, for each $t\le 1$ (resp. all $t>0$, if $K\ge 0$),
\begin{equation}\label{volume-com}
\mu(B(x,{\sqrt t}))\le C\exp\lf(\frac{Cd(x,y)}{\sqrt{t}}\r)\mu(B(y,{\sqrt t})).
\end{equation}
Indeed, if $K\ge 0$, then by Lemma \ref{volume-growth}(i), we have that for all $t>0$,
\begin{eqnarray*}
\mu(B(x,{\sqrt t}))&&\le \mu\left(B(y,{d(x,y)+\sqrt t})\right)\le \lf(\frac{d(x,y)+\sqrt t}{\sqrt t}\r)^N\mu(B(y,{\sqrt t}))\\
&&\le  C\exp\lf(\frac{Cd(x,y)}{\sqrt{t}}\r)\mu(B(y,{\sqrt t})).
\end{eqnarray*}
If $K<0$, $t\le 1$ and $d(x,y)\le 2\sqrt t$, then by the local doubling condition Lemma \ref{volume-growth}(ii), we see that
\begin{eqnarray*}
\mu(B(x,{\sqrt t}))\le \mu\left(B(y,3\sqrt t)\right)\le C(K,N)\mu(B(y,{\sqrt t}))\le  C\exp\lf(\frac{Cd(x,y)}{\sqrt{t}}\r)\mu(B(y,{\sqrt t})).
\end{eqnarray*}
If $K<0$, $t\le 1$ and $d(x,y)>2\sqrt t$, then we choose a geodesic $\gz$ connecting $x$ to $y$, which satisfies $length(\gz)=d(x,y)$.
Taking the largest natural number $C_\gz$ smaller than $\frac{d(x,y)}{\sqrt t}+1$ and, dividing the curve $\gz$ into $C_\gz$ pieces
of equal length, we obtain a sequence of points $\{x_i\}_{i=0}^{C_\gz}$ with $x_0=x$, $x_{C_\gz}=y$ and $d(x_i,x_{i+1})=\frac{d(x,y)}{C_\gz}\le \sqrt t$.
Applying the local doubling condition Lemma \ref{volume-growth}(ii) $C_\gz$ times, we obtain
\begin{eqnarray*}
\mu(B(x,{\sqrt t}))&&\le \mu\left(B\left(x_1,2\sqrt t\right)\right)\le C(K,N)\mu(B(x_1,{\sqrt t}))\\
&&\le  \cdots\le C(K,N)^{C_\gz}\mu(B(y,{\sqrt t}))\le  C\exp\lf(\frac{Cd(x,y)}{\sqrt{t}}\r)\mu(B(y,{\sqrt t})).
\end{eqnarray*}

Hence, \eqref{hk} and \eqref{volume-com} imply that, for each $t\le 1$ (resp. all $t>0$, if $K\ge 0$) and all $x,y\in X$, it holds
\begin{eqnarray}\label{hk2}
  &&C^{-1}\mu(B(x,{\sqrt t}))^{-1}\exp\lf\{-\frac{d(x,y)^2}{C_2t}\r\}\le p(t,x,y)\le
  C\mu(B(x,{\sqrt t}))^{-1}\exp\lf\{-\frac{d(x,y)^2}{C_1t}\r\}.
\end{eqnarray}

\section{Some a priori heat kernel estimates}
\hskip\parindent In this section, we  establish the mapping property for the operator $|\nabla H_t|$.

The following gradient estimate was established in \cite[Thoerem 3.1]{jky14} by choosing the natural Dirichlet energy on $RCD^\ast(K,N)$ spaces and using the gradient estimates of harmonic functions from \cite{ji14}.

\begin{thm}\label{gpe}
Let $(X,d,\mu)$ be  a $RCD^\ast(K,N)$ space, where $K\in \rr$ and $N\in [1,\infty)$.
Let $\Omega\subset X$ and suppose that $\Delta u=g$ in $\Omega$ with $g\in L^\fz(\Omega)$.  Then

{\rm (i)} if $K\ge 0$, there exists $C(N)>0$ such that, for every ball $B=B(x_0,R)$
with $2B\subset\subset \Omega$ and almost every $x\in B$, it holds
\begin{equation}\label{gradient-lap}
|\nabla u(x)|\le C(N)\lf\{\frac{1}{R}\fint_{2B}|u|\,d\mu
+\sum_{j=-\fz}^{\lfr \log_2
R\rf}2^j\lf(\fint_{B(x,2^j)}|g|^{p_N}\,d\mu\r)^{1/p_N}\r\},
\end{equation}
 where
$p_N=1$ when $1<N<2$, $p_2=3/2$, and
$p_N=\frac {2N}{N+2}$ when $N>2$.

{\rm (ii)} if $K<0$, \eqref{gradient-lap} holds for every ball $B=B(x_0,R)$
with $2B\subset\subset \Omega$ and $R\le 1$, and with $C(N)>0$ replaced by $C(K,N)$ which depends also on $K$.
\end{thm}

Since, for each $t>0$, the heat kernel is a solution to the equation $\Delta p_t=\frac{\partial}{\partial t}p_t$,
we may apply the above gradient estimates to heat kernels. The following result was established in \cite{jkyz14},
we report it here for completeness.

\begin{thm}\label{heatkernel}
Let $(X,d,\mu)$ be  a $RCD^\ast(K,N)$ space, where $K\in \rr$ and $N\in [1,\infty)$.
If $K\ge 0$, then there exists $c,C(N)>0$ such that, for almost all $x,y\in X$ and $t>0$, it holds that
\begin{equation}\label{ghk}
\lf|\nabla_y p_t(x,y)\r|\le C(N)
\frac1{\sqrt t \mu(B(x,\sqrt t))}\exp\lf\{-\frac{d(x,y)^2}{ct}\r\}.
\end{equation}

If $K<0$, then \eqref{ghk} holds for almost all $x,y\in X$ and $t\in (0,1]$ with
$C(N)$ replaced by $C(N,K)$, which depends on $N,K$.
\end{thm}
\begin{proof}
Notice that, for each $t>0$, the heat kernel is a solution to the heat equation
$$\Delta p_t=\frac{\partial}{\partial t}p_t.$$

Using the estimates for time differentials of heat kernels from Sturm \cite[Theorem 2.6]{st2}
and \eqref{volume-com}, it follows that, for almost all $x,y\in X$, it holds that
\begin{equation}\label{thk}
\lf|\frac{\partial}{\partial t}p_t(x,y)\r|\le \frac Ct
\frac1{\mu(B(x,\sqrt t))}\exp\lf\{-\frac{d(x,y)^2}{ct}\r\}
\end{equation}
for each $t>0$ if $K\ge 0$, and for each $t\in (0,1]$ if $K<0$.

Fix a $t>0$. If $K<0$, we additionally require $t\le 1$. Notice that, for each such fixed
$t>0$, $\frac{\partial}{\partial t}p_t$ is a locally bounded function. Thus, by Theorem \ref{gpe}, we see that, for almost all $x,y$,
\begin{eqnarray*}
|\nabla _yp_t(x,y)|&&\le C\lf\{\frac{1}{\sqrt t}\fint_{B(y,{2\sqrt t})}p_t(x,z)\,d\mu(z)\r. \\
&& \lf.+\sum_{j=-\fz}^{\lfr\log_2 {2\sqrt
t}\rf}2^j\lf(\fint_{B(y,2^j)}\lf|\frac{\partial}{\partial
t}p_t(x,z)\r|^{p_N}\,d\mu\r)^{1/p_N}\r\}.
\end{eqnarray*}

We divide the proof into two cases, i.e., $d(x,y)^2>16t$ and
$d(x,y)^2\le 16t$. If $d(x,y)^2>16t$, then  it holds, for every $z\in B_{2\sqrt t}(y)$, that
$$d(x,z)\ge d(x,y)-d(y,z)\ge  d(x,y)- d(x,y)/2= d(x,y)/2.$$
From this, together with \eqref{hk2}, we see that
\begin{eqnarray*}
\frac{1}{\sqrt t}\fint_{B(y,{2\sqrt t})}p_t(x,z)\,d\mu(z)&&\le
\frac {C}{\sqrt t  \mu(B(y,2\sqrt t))}\int_{B(y,2\sqrt t)}\frac {e^{-\frac {d(x,z)^2}{ct}}}{{\mu(B(x,\sqrt t))}}\,d\mu(z)\\
&&\le\frac {C}{\sqrt t  \mu(B(x,\sqrt t))}e^{-\frac {d(x,y)^2}{ct}}.
\end{eqnarray*}

 By using \eqref{thk} and the doubling condition, we also obtain
\begin{eqnarray*}
&&\sum_{j=-\fz}^{\lfr\log_2 {2\sqrt
t}\rf}2^j\lf(\fint_{B(y,2^j)}\lf|\frac{\partial}{\partial
t}p(t,x,z)\r|^{p_N}\,d\mu\r)^{1/p_N}\\
&&\quad\le\sum_{j=-\fz}^{\lfr\log_2 {2\sqrt
t}\rf}2^j\lf(\fint_{B(y,2^j)} \lf|\frac Ct \frac1{\mu(B(x,\sqrt t))}\exp\lf\{-\frac{d(x,z)^2}{ct}\r\}
\r|^{p_N}\,d\mu(z)\r)^{1/p_N}\\
&&\quad\le \frac Ct \frac1{\mu(B(x,\sqrt t))}\exp\lf\{-\frac{d(x,y)^2}{ct}\r\} \sum_{j=-\fz}^{\lfr\log_2 {2\sqrt t}\rf}2^j\\
&&\quad\le \frac {C}{\sqrt t}  \frac1{\mu(B(x,\sqrt t))}\exp\lf\{-\frac{d(x,y)^2}{ct}\r\}.
\end{eqnarray*}

Combining the above two estimates we conclude that, for $\mu$-a.e. $x,y$,
\begin{equation}\label{5.5}
\lf|\nabla_y p_t(x,y)\r|\le \frac {C}{\sqrt t}  \frac1{\mu(B(x,\sqrt t))}\exp\lf\{-\frac{d(x,y)^2}{ct}\r\}.
\end{equation}

When $x,y\in X$ with $d(x,y)^2\le 16t$, the exponential term
$\exp\{-\frac{d(x,y)^2}{ct}\}$ is equivalent to 1. Applying the
proof for \eqref{5.5} and discarding  the exponential term, we
arrive at
\begin{equation*}
\lf|\nabla_y p_t(x,y)\r|\le \frac {C}{\sqrt t}  \frac1{\mu(B(x,\sqrt t))},
\end{equation*}
which, together with \eqref{5.5}, implies that for almost all $x,y\in
X$,
\begin{equation*}
\lf|\nabla_y p_t(x,y)\r|\le \frac {C}{\sqrt t}  \frac1{\mu(B(x,\sqrt t))}\exp\lf\{-\frac{d(x,y)^2}{ct}\r\}.
\end{equation*}
%
%
The proof is then completed.
\end{proof}

Based on the gradient estimates of heat kernels, Theorem \ref{heatkernel} and \eqref{thk}, we conclude the
following mapping property of $|\nabla H_t|$.  We next summarize this and some results from Sturm \cite{st2} as follows.
Let $1\le p,q\le \infty$. For an operator $T$, we denote its operator norm
from $L^p(X)$ to $L^q(X)$ by $\|T\|_{p,q}$.

\begin{thm}\label{map-hk}
Let $(X,d,\mu)$ be  a $RCD^\ast(K,N)$ space, where $K\in \rr$ and $N\in [1,\infty)$.

{\rm (i)} For each $t>0$ and $p\in [1,\infty]$, the operator $H_t$ is bounded on $L^p(X)$ with $\|H_t\|_{p,p}\le 1$.

{\rm (ii)} If $K\ge 0$, then, for each $t>0$,
the operators $\sqrt t|\nabla H_t|$ and $t\Delta H_t$ are bounded on $L^p(X)$ for all $p\in [1,\infty]$.
Moreover, there exists $C>0$, such that, for all $t>0$ and  all $p\in [1,\infty]$,
$$\max\lf\{ \|\sqrt t|\nabla H_t|\|_{p,p}, \|t\Delta H_t\|_{p,p}\r\}\le C.$$

{\rm (iii)} If $K< 0$, then, for each $t>0$,
the operators $\sqrt t|\nabla H_t|$ and $t\Delta H_t$ are bounded on $L^p(X)$ for all $p\in [1,\infty]$.
Moreover, there exists $C>0$, such that, for all $t>0$ and all $p\in [1,\infty]$,
$$\max\lf\{ \|\sqrt {(t\wedge 1)}|\nabla H_t|\|_{p,p}, \|(t\wedge 1)\Delta H_t\|_{p,p}\r\}\le C.$$
\end{thm}
\begin{proof}
(i) was obtained by Sturm \cite{st2}.  To prove (ii), we use Theorem \ref{heatkernel}, Lemma \ref{volume-growth} and \eqref{hk2} to see that,
for each $f\in L^p(X)$, $p\in [1,\infty]$ and $\mu$-a.e. $x\in X$, it holds
\begin{eqnarray*}
\sqrt t|\nabla H_t(f)(x)|&&\le C(N)\int_X
\frac1{ \mu(B(x,\sqrt t))}\exp\lf\{-\frac{d(x,y)^2}{ct}\r\}|f(y)|\,d\mu(y)\le CH_{ct}(|f|)(x)
\end{eqnarray*}
for all $t>0$. Hence, by using (i), we obtain that
\begin{eqnarray*}
\|\sqrt t|\nabla H_t(f)(x)|\|_{L^p(X)}&&\le C\|H_{ct}(|f|)\|_{L^p(X)}\le C\|f\|_{L^p(X)}.
\end{eqnarray*}
Similar calculations using \eqref{thk} give the desired conclusions for $t\Delta H_t$.

Let us prove (iii). If $t\le 1$, then the same arguments of (ii) yield that
$$\max\lf\{ \|\sqrt t|\nabla H_t|\|_{p,p}, \|t\Delta H_t\|_{p,p}\r\}\le C.$$
If $t>1$, then, by using the property of semigroup and  the $L^p$-boundedness of $|\nabla H_1|,\,\Delta H_1,\,H_{t-1}$, we obtain
\begin{eqnarray*}
\||\nabla H_t|\|_{p,p}+\|\Delta H_t\|_{p,p}= \||\nabla H_1(H_{t-1})|\|_{p,p}+\|\Delta H_1(H_{t-1})\|_{p,p}\le C\|H_{t-1}\|_{p,p}\le C,
\end{eqnarray*}
which completes the proof.
\end{proof}

If $(X,d,\mu)$ is a compact $RCD^\ast(K,N)$ space, then from Lemma \ref{volume-growth}, we know that, for each $t\in (0,\infty)$, it holds
$$\frac1{\mu(B(x,\sqrt t))}\le \max\lf\{\frac{\ell_{K,N}(\mbox{diam}(X))}{\ell_{K,N}(\sqrt t)\mu(X)},\,\frac{(\mbox{diam}(X))^N}{(\sqrt t)^N\mu(X)},\,\frac{1}{\mu(X)}\r\}.$$
From this and \eqref{hk2}, one can deduce that for each $t>0$, $p_t$ is bounded on $X\times X$, and hence, $H_t$ is bounded from
$L^1(X)$ to $L^\infty(X)$.

However,  if $(X,d,\mu)$ is non-compact, then $H_t$ may not be bounded from $L^q(X)$ to $L^\infty(X)$ for any $q\in [1,\infty)$.
Indeed, we have the following result.

\begin{prop}\label{q-infty-bound}
Let $(X,d,\mu)$ be  a $RCD^\ast(K,N)$ space, where $K\in \rr$ and $N\in [1,\infty)$. Then the following conditions are equivalent.

{\rm (i)} $\inf_{x\in X}\mu(B(x,1))=C_{X}>0$;

{\rm (ii)} For each $t>0$, $H_t$ is bounded from $L^q(X)$ to $L^\infty(X)$ for all $q\in [1,\infty)$;

{\rm (iii)} For each $t>0$, $H_t$ is bounded from $L^q(X)$ to $L^\infty(X)$ for some $q\in [1,\infty)$.
\end{prop}
\begin{proof} Let us show that (i) implies (ii). Suppose $\inf_{x\in X}\mu(B(x,1))=C_X>0$. Fix a $t\in (0,1].$ Then Lemma \ref{volume-growth} implies  that $t^{N/2}C_X\le \mu(B(x,\sqrt t))$ if $K\ge 0$, and
$$\frac1{\mu(B(x,\sqrt t))}\le \frac{\ell_{K,N}(1)}{\ell_{K,N}(\sqrt {t})\mu(B(x,1))}\le \frac{\ell_{K,N}(1)}{\ell_{K,N}(\sqrt {t})C_X},$$
if $K<0$.
These together with \eqref{hk2}, imply that for all $x,y\in X$,
$$p_t(x,y)\le \frac{C(K,N,t)}{C_X}\exp\lf\{-\frac{d(x,y)^2}{ct}\r\}.$$
Hence $p_t\in L^\infty(X\times X)$, and $\|H_t\|_{1,\infty}\le C(K,N,t)<\infty$. If $t>1$, then by using the  $L^1\to L^\infty$-boundedness of
$H_1$ and $L^1$-boundedness of $H_{t-1}$, we find that for each $f\in L^1(X)$ it holds
$$\|H_tf\|_{L^\infty(X)}=\|H_1\circ H_{t-1}f\|_{L^\infty(X)}\le C(K,N)\|H_{t-1}f\|_{L^1(X)}\le C(K,N)\|f\|_{L^1(X)},$$
i.e., $\|H_t\|_{1,\infty}\le C(K,N)$ for $t>1$.
Since $\|H_t\|_{\infty,\infty}\le 1$, we conclude that $H_t$ is bounded from $L^q(X)$ to $L^\infty(X)$ for all $q\in [1,\infty]$.

It is obvious that (ii) implies (iii). Let us prove that (iii) implies (i).
Suppose $\|H_t\|_{q,\infty}\le C$ for some $q\in (1,\infty)$. A duality argument shows $\|H_t\|_{1,q'}\le C$,
where $q'$ is the H\"older conjugate of $q$. Then for $f\in L^{q'}(X)$, the H\"older inequality and
the fact $H_t1=1$ imply that for each $x\in X$,
 $$|H_tf(x)|\le \lf(\int_Xp_t(x,y)\,d\mu\r)^{1/q}\lf(\int_Xp_t(x,y)|f(y)|^{q'}\,d\mu\r)^{1/q'}= |H_t(|f|^{q'})(x)|^{1/q'},$$
 and hence,
\begin{eqnarray*}
\|H_tf\|_{L^{(q')^2}(X)}=\lf(\int_X |H_tf|^{(q')^2}\,d\mu\r)^{1/(q')^2}\le \lf(\int_X |H_t(|f|^{q'})|^{q'}\,d\mu\r)^{1/(q')^2}\le C\||f|^{q'}\|^{1/q'}_{L^{1}(X)}=C\|f\|_{L^{q'}(X)}.
\end{eqnarray*}
Hence, $\|H_t\|_{q',(q')^2}\le C$ and $\|H_t\|_{1,(q')^2}\le C$. Repeating this argument $k$ times, where
$(q')^k\ge 2$, it follows  $\|H_t\|_{1,(q')^k}\le C$. Since $\|H_t\|_{1,1}\le 1$, we then see that
$\|H_t\|_{1,2}\le C$. Using a duality argument again, we conclude that $\|H_t\|_{2,\infty}\le C$,
and hence $\|H_t\|_{1,\infty}\le C$.

Thus, to finish the proof, we only need to show that $\|H_t\|_{1,\infty}\le C$ implies $\inf_{x\in X}\mu(B(x,1))>0$. Let $t=1$.
 For each $x_0\in X$, choose a function $0\le f\in L^1(X)$ satisfying $\supp f\subset B(x_0,1)$ and $\|f\|_{L^1(X)}=1$.  Then the heat kernel estimate \eqref{hk2} yields
$$H_1(f)(x_0)\ge C\mu(B(x_0,1))^{-1}\int_{B(x_0,1)}f(y)\,d\mu(y)\ge C\mu(B(x_0,1))^{-1},$$
which implies that
$$\|H_1\|_{1,\infty}=\sup_{\|f\|_{L^1(X)\le 1}} \|H_1f\|_{L^\infty(X)}\ge C\mu(B(x_0,1))^{-1}$$
for each $x_0\in X$. Taking supremum over $x_0\in X$, we see that
$$\sup_{x_0\in X}\mu(B(x_0,1))^{-1}\le C\|H_1\|_{1,\infty}<\infty,$$
which is equivalent to say $\inf_{x\in X}\mu(B(x,1))=C_{X}>0$. The proof is therefore completed.
\end{proof}

Since there exists a complete Riemannian manifold $M$, with non-negative Ricci curvature,
that satisfies $\inf_{x\in M}\mu(B(x,1))=0$ (see \cite{ck88}), we may loose the global upper bound for
$H_tf$, and therefore, we do not know that if $|\nabla H_tf|$ has a global upper bound.
In these cases,  we have the following local bounds, which will be useful in proving the main results.

\begin{lem}\label{local-bound}
Let $(X,d,\mu)$ be  a $RCD^\ast(K,N)$ space, where $K\in \rr$ and $N\in [1,\infty)$. Let $f\in L^q(X)$ for some $q\in [1,\infty)$.
Then, for each $t>0$, $|\nabla  H_tf|$, $(\Delta H_tf)$ and $H_tf$ are locally bounded functions. More precisely,
for each $B(x_0,r)\subset X$, $r\ge 1$, it holds
\begin{equation}
\||\nabla  H_tf|+|\Delta H_tf|+|H_tf|\|_{L^\infty(B(x_0,r))}\le \frac{C(K,N,t,r,q)}{\mu(B(x_0,r))^{1/q}}\|f\|_{L^q(X)}.
\end{equation}
\end{lem}
\begin{proof}
Suppose first $t\le 1$. By Lemma \ref{volume-growth} and decomposing the integral in diadic annuli, we conclude that
for any $x\in X$ and $c_0>0$, it holds
\begin{equation}\label{exp-vol}
  \int_X\exp\lf\{-c_0d(x,y)^2\r\}\,d\mu(y)\le C(c_0,K,N)\mu(B(x,1)).
\end{equation}

By \eqref{exp-vol}, Theorem \ref{ghk} and the H\"older inequality, we conclude that, for $\mu$-a.e. $x\in B(x_0,r)$,
\begin{eqnarray*}
|\nabla  H_tf(x)|&&\le \frac {C(N,K)}{\sqrt t}  \frac1{\mu(B(x,\sqrt t))}\int_X|f(y)|\exp\lf\{-\frac{d(x,y)^2}{ct}\r\}\,d\mu(y)\\
&& \le \frac {C(N,K)}{\sqrt t}  \frac{1}{\mu(B(x,\sqrt t))}\|f\|_{L^q(X)}C(N,K,q,t)\mu(B(x,1))^{(q-1)/q}\\
&&\le C(N,K,q,t)\frac{\ell_{K,N}(2r)^{1+1/q}}{\ell_{K,N}(\sqrt t)\ell_{K,N}(1)^{1/q}\mu(B(x,2r))^{1/q}}\|f\|_{L^q(X)}\\
&& \le \frac{C(N,K,q,t,r)}{\mu(B(x_0,r))^{1/q}}\|f\|_{L^q(X)}.
\end{eqnarray*}

Using \eqref{hk2} and \eqref{thk}, instead of Theorem \ref{ghk}, the same estimate as the above inequality yields
\begin{eqnarray*}
\||\nabla  H_tf|+|\Delta H_tf|+|H_tf|\|_{L^\infty(B(x_0,r))}&& \le \frac{C(N,K,q,t,r)}{\mu(B(x_0,r))^{1/q}}\|f\|_{L^q(X)},
\end{eqnarray*}
as desired.

If $t>1$, then by Theorem \ref{map-hk}(i), we find
\begin{eqnarray*}
&&\||\nabla  H_tf|+|\Delta H_tf|+|H_tf|\|_{L^\infty(B(x_0,r))}\\
&&\quad\le \||\nabla  H_1(H_{t-1}f)|+|\Delta H_1(H_{t-1}f)|+|H_1(H_{t-1}f)|\|_{L^\infty(B(x_0,r))} \le \frac{C(N,K,q,t,r)}{\mu(B(x_0,r))^{1/q}}\|f\|_{L^q(X)},
\end{eqnarray*}
 which completes the proof.
\end{proof}

\section{Generalized Bochner inequality}
\hskip\parindent
In this section, we give a generalization of the Bochner inequality (Theorem \ref{Bochner}).
We note that the main result (Theorem \ref{gbi}) is already obtained in \cite{ams13} under milder assumptions,
we keep the  proofs  here for completeness.

We shall need the following results on the existence of cut-off functions from  \cite[Lemma 6.7]{ams13};
see also \cite{hkx13,gm14}.

\begin{lem}\label{cut-off}
Let $(X,d,\mu)$ be  a $RCD^\ast(K,N)$ space, where $K\in \rr$ and $N\in [1,\infty)$. Let $x_0\in X$ be fixed.
Then

{\rm (i)} for each $0<r<\infty$, there exists $\phi\in W^{1,2}(X)\cap LIP(X)\cap L^\infty(X)$  satisfying $\phi=1$ on $B(x_0,r)$ and
$\phi=0$ on $X\setminus B(x_0,r+1)$, $|\nabla \phi|\le C$ and $\|\Delta \phi\|_{L^\infty(X)}\le  C$;

{\rm (ii)} for each $0<r<\infty$, there exists a Lipschitz cut-off function $\Phi$ satisfying
$\Phi=1$ on $B(x_0,r)$, $\Phi,\Delta\Phi\in W^{1,2}{(X)}\cap L^\infty(X)$ with compact support.
\end{lem}

The following $L^\infty$-estimates can be found in \cite[Theorem 3.1]{ams13}, we give a proof for completeness.
\begin{lem}\label{test-lip}
Let $(X,d,\mu)$ be  a $RCD^\ast(K,N)$ space, where $K\in \rr$ and $N\in [1,\infty)$. Suppose
$g\in \mathcal{D}(\Delta)\cap L^\infty(X)$ with $\Delta g\in L^{\infty}(X)$. Then there exists a constant $C=C(K,N)>0$
such that
$$\||\nabla g|\|_{L^\infty(X)}\le C\lf[\|g\|_{L^\infty(X)}+\|\Delta g\|_{L^\infty(X)}\r].$$
\end{lem}
\begin{proof}
 Since $g\in \mathcal{D}(\Delta)\cap L^\infty(X)$ with $\Delta g\in L^{\infty}(X)$, by applying the gradient estimate (Theorem \ref{gpe}) for each ball
$B(x_0,1)\subset X$, it follows, for almost every $x\in B(x_0,1)$, that
\begin{eqnarray*}
|\nabla g(x)|&&\le C(K,N)\lf\{\fint_{B(x_0,2)}|g|\,d\mu
+\sum_{j=-\fz}^{0}2^j\lf(\fint_{B(x,2^j)}|\Delta g|^{p_N}\,d\mu\r)^{1/p_N}\r\}\\
&&\le C(K,N)\lf[\|g\|_{L^\infty(B(x_0,2))}+\|\Delta g\|_{L^\infty(B(x_0,2))}\r]\\
&&\le C(K,N)\lf[\|g\|_{L^\infty(X)}+\|\Delta g\|_{L^\infty(X)}\r].
\end{eqnarray*}
This implies that
$$\||\nabla g|\|_{L^\infty(X)}\le C(K,N)\lf[\|g\|_{L^\infty(X)}+\|\Delta g\|_{L^\infty(X)}\r],$$
which completes the proof.
\end{proof}

We next consider the generalized Bochner inequality.
 As pointed out at the beginning of this section,  \cite[Corollary 4.3]{ams13} actually provides a stronger
result, we choose to give a proof below for completeness.
\begin{thm}[Generalized Bochner Inequality]\label{gbi}
Let $(X,d,\mu)$ be  a $RCD^\ast(K,N)$ space, where $K\in \rr$ and $N\in [1,\infty)$.
Suppose that $f\in W^{1,2}_{\mathrm{loc}}(X)\cap L^\infty_{\mathrm{loc}}(X)$ satisfies  $|\nabla f|\in L^2(X)$ and $\Delta f\in W^{1,2}(X)$, and $\la \nabla f,\nabla \Phi\ra\in W^{1,2}(X)$
for each $\Phi$ satisfying $\Phi,\Delta\Phi\in W_c^{1,2}{(X)}\cap L^\infty(X)$.

Then, for all $0\le g\in \mathcal{D}(\Delta)\cap L^\infty(X)$ with $\Delta g\in L^{\infty}(X)$, it holds
\begin{equation}
  \frac 12 \int_X \Delta g|\nabla f|^2\,d\mu- \int_X g\langle \nabla f,\nabla\Delta f\rangle\,d\mu\ge K\int_X g|\nabla f|^2\,d\mu +\frac{1}{N}\int_X g(\Delta f)^2\,d\mu.
\end{equation}
\end{thm}
\begin{proof}
The conclusion is obvious if $(X,d)$ is compact, let us consider the remaining cases.

Let $x_0\in X$ be fixed. For each $k\in \mathbb N$, by Lemma \ref{cut-off}(i),
there exists $\phi_k$ satisfying $\phi_k=1$ on $B(x_0,k)$ and
$\phi_k=0$ on $X\setminus B(x_0,k+1)$, $|\nabla \phi_k|\le C$ and $\|\Delta \phi_k\|_{L^\infty(X)}\le C$

Moreover, by Lemma \ref{cut-off}(ii),  for each $k\in \mathbb N$, there exists a cut-off function $\Phi_k$ satisfying
$\Phi_k=1$ on $B(x_0,k+1)$, $\Phi_k,\Delta\Phi_k\in W^{1,2}{(X)}\cap L^\infty(X)$ with compact supports.

Notice that, by the choices of cut-off functions, we have $\Phi_k f\in W^{1,2}(X)$. Moreover,  $\Phi_k\in \mathcal {D}(\Delta)\cap L^\infty(X)$ and $f\in \mathcal{D}_{\mathrm {loc}}(\Delta)\cap L^\infty_{\mathrm {loc}}(X)$, it follows, from the Leibniz rule, that
$\Phi_kf\in \mathcal{D}(\Delta)$ and $\Delta (\Phi_k f)=f\Delta\Phi_k +\Phi_k \Delta f+2\nabla\Phi_k\cdot \nabla f$.
 Since $\Phi_k,\Delta\Phi_k\in W^{1,2}_c{(X)}\cap L^\infty(X)$,  $\nabla\Phi_k\cdot \nabla f\in W^{1,2}(X)$ by the assumption,
we see that $\Delta (\Phi_k f)\in W^{1,2}(X)$.

On the other hand, notice that $g\phi_k\in \mathcal{D}(\Delta)\cap L^\infty(X)$ and $\Delta (g\phi_k)=\phi_k\Delta g+g\Delta \phi_k+2\nabla g\cdot \nabla \phi_k\in L^\infty$. Theorem \ref{Bochner} then implies that, for each $k\in \mathbb N$, it holds
\begin{eqnarray*}
 && \frac 12 \int_X \Delta (g\phi_k)|\nabla (\Phi_k f)|^2\,d\mu- \int_X (g\phi_k)\langle \nabla(\Phi_k f),\nabla\Delta(\Phi_k f)\rangle\,d\mu\\
&&\hs \ge K\int_X (g\phi_k )|\nabla (\Phi_k f)|^2\,d\mu +\frac{1}{N}\int_X (g\phi_k)(\Delta(\Phi_k f))^2\,d\mu.
\end{eqnarray*}

Since $\supp \phi_k, \supp |\nabla \phi_k|, \supp \Delta\phi_k\subset \overline{B(x_0,k+1)}$ and $\Phi_k=1$ on $\overline{B(x_0,k+1)}$, 
the above inequality reduces to
\begin{eqnarray}\label{3g-bochner}
\quad\quad && \frac 12 \int_X \Delta (g\phi_k)|\nabla f|^2\,d\mu- \int_X (g\phi_k)\langle \nabla f,\nabla\Delta f\rangle\,d\mu\ge K\int_X (g\phi_k )|\nabla  f|^2\,d\mu +\frac{1}{N}\int_X (g\phi_k)(\Delta f)^2\,d\mu.
\end{eqnarray}

The choices of $\phi_k$ further imply that, for each $k$,
\begin{eqnarray*}
\lf|\int_X \Delta (g\phi_k)|\nabla f|^2\,d\mu-\int_X \Delta g|\nabla f|^2\,d\mu\r| &&  \le \int_X \lf|[\Delta g(\phi_k-1)+2\la \nabla g,\nabla \phi_k\ra+g\Delta\phi_k]\r| |\nabla f|^2\,d\mu\\
&&\le \int_{X\setminus B(x_0,k)} \lf[|\Delta g|+2|\nabla g|+|g|\r] |\nabla f|^2\,d\mu.
\end{eqnarray*}
This, together with Lemma \ref{test-lip}, implies that
\begin{eqnarray*}
\lf|\int_X \Delta (g\phi_k)|\nabla f|^2\,d\mu-\int_X \Delta g|\nabla f|^2\,d\mu\r| &&
\le C(K,N)\lf[\|g\|_{L^\infty(X)}+\|\Delta g\|_{L^\infty(X)}\r] \int_{X\setminus B(x_0,k)} |\nabla f|^2\,d\mu,
\end{eqnarray*}
which tends to zero as $k\to\infty$, since $|\nabla f|\in L^2(X)$.

By  a similar but easier argument and letting $k\to\infty$ in \eqref{3g-bochner}, we obtain
\begin{equation}
  \frac 12 \int_X \Delta g|\nabla f|^2\,d\mu-\int_X g\langle \nabla f,\nabla\Delta f\rangle\,d\mu\ge K\int_X g|\nabla f|^2\,d\mu +\frac{1}{N}\int_X g(  \Delta f)^2\,d\mu,
\end{equation}
which completes the proof.
\end{proof}

We next show that the Bochner inequality holds for our main target, $\log H_t (f+\delta)$ where $0\le f\in L^1(X)\cap L^\infty(X)$ (see the following section).

To this end, we need the following self-improvement property  proved by Savar\'e \cite{sa14}.

\begin{lem}\label{self-improvement}
Let $(X,d,\mu)$ be  a $RCD^\ast(K,N)$ space, where $K\in \rr$ and $N\in [1,\infty)$.
If $f\in W^{1,2}(X)\cap LIP(X)\cap L^\infty(X)$ satisfying $\Delta f\in W^{1,2}(X)$, then $|\nabla f|^2\in W^{1,2}(X)\cap L^\infty(X)$.
\end{lem}

The following lemma shows that $\log (H_t f+\delta)$ satisfies the requirements of Theorem \ref{gbi}. In what follows,
we write $H_tf+\delta$ as  $H_t(f_\delta)$.
\begin{lem}\label{log-function}
Let $(X,d,\mu)$ be  a $RCD^\ast(K,N)$ space, where $K\in \rr$ and $N\in [1,\infty)$.
Then, for all $0\le f\in L^1(X)\cap L^\infty(X)$ and $s,\delta>0$, the following holds:

{\rm (i)} $\log H_s(f_\delta)\in W^{1,2}_{\mathrm{loc}}(X)\cap L^\infty(X)$ and $|\nabla \log H_s(f_\delta)|\in L^2(X)\cap L^\infty(X)$;

 {\rm (ii)} $\Delta \log H_s(f_\delta)\in W^{1,2}(X)\cap L^\infty(X)$;

 {\rm (iii)} $\la \nabla \log H_s(f_\delta),\nabla \Phi\ra\in W^{1,2}(X)\cap L^\infty(X)$ for each $\Phi$ satisfying $\Phi,\Delta\Phi\in W_c^{1,2}{(X)}\cap L^\infty(X)$.
\end{lem}
\begin{proof} {(i)} Notice that, by the mapping properties of $H_s$ and $|\nabla H_s|$ (Theorem \ref{map-hk}), we have
$$H_sf\in W^{1,2}(X)\cap L^\infty(X)\cap LIP(X)$$
 for each $s>0$.
Hence $\log H_s(f_\delta)=\log (H_sf+\delta)\in W^{1,2}_{\mathrm{loc}}(X)\cap LIP(X)\cap L^\infty(X)$ and, from the chain rule, it follows that
$$\lf|\nabla \log H_s(f_\delta)\r|=\frac{|\nabla H_s f|}{H_s(f_\delta)}\in L^2(X)\cap L^\infty(X).$$

{(ii)} Using the chain rule, it follows that 
$\log H_s(f_\delta)\in \mathcal{D}_{\mathrm{loc}}(\Delta)$ and
$$\Delta (\log H_s(f_\delta))=\frac{\Delta H_sf}{H_s(f_\delta)}-\frac{|\nabla H_sf|^2}{(H_s(f_\delta))^2}.$$
From Theorem \ref{map-hk} and Lemma \ref{self-improvement}, we deduce that $|\nabla H_sf|^2\in W^{1,2}(X)\cap L^\infty(X)$ and
$$\Delta H_sf\in W^{1,2}(X)\cap L^\infty(X)\cap LIP(X).$$
These further imply that $\Delta (\log H_s(f_\delta))\in W^{1,2}(X)\cap L^\infty(X)$.

{(iii)} Let $\Phi$ satisfy $\Phi,\Delta\Phi\in W_c^{1,2}{(X)}\cap L^\infty(X)$. Then
$$\la \nabla \log H_s(f_\delta),\nabla \Phi\ra=\frac{1}{H_s(f_\delta)}\la \nabla H_sf,\nabla \Phi\ra
=\frac 1{4H_s(f_\delta)}\lf[|\nabla (\Phi+H_sf)|^2-|\nabla(\Phi-H_sf)|^2\r].$$
Applying  Lemma \ref{self-improvement} and Lemma \ref{test-lip}, we conclude that
$$|\nabla (\Phi+H_sf)|^2,|\nabla(\Phi-H_sf)|^2\in W^{1,2}(X)\cap L^\infty(X).$$
Since $|\nabla \frac{1}{H_s(f_\delta)}|=\frac{|\nabla H_sf|}{(H_s(f_\delta))^2}\in L^2(X)\cap L^\infty(X)$,
we finally see that $\la \nabla \log H_s(f_\delta),\nabla \Phi\ra\in W^{1,2}(X)$.
\end{proof}

\begin{cor}\label{bi-log}
Let $(X,d,\mu)$ be  a $RCD^\ast(K,N)$ space, where $K\in \rr$ and $N\in [1,\infty)$. Let
$0\le g\in \mathcal{D}(\Delta)\cap L^\infty(X)$ with $\Delta g\in L^{\infty}(X)$.
Then, for all $0\le f\in L^1(X)\cap L^\infty(X)$ and $s,\delta>0$, it holds
\begin{eqnarray}\label{b-log}
\quad&&  \frac 12 \int_X \Delta g|\nabla (\log H_s(f_\delta))|^2\,d\mu- \int_X g\langle \nabla(\log H_s(f_\delta)),\nabla \Delta (\log H_s(f_\delta))\rangle\,d\mu\\
&&\quad\ge K\int_X g|\nabla (\log H_s(f_\delta))|^2\,d\mu +\frac{1}{N}\int_X g(\Delta (\log H_s(f_\delta)))^2\,d\mu,\nonumber
\end{eqnarray}
where we set $f_\delta :=f+\delta$.
\end{cor}
\begin{proof} Lemma \ref{log-function} implies that $\log H_s(f_\delta)$ satisfies the requirements of
Theorem \ref{gbi}, and hence the corollary follows directly.
\end{proof}

\section{The Li-Yau inequality}
\hskip\parindent
The main aim of this section is to prove the Li-Yau inequality (Theorem \ref{lye}) for solutions
to the heat equation on $RCD^\ast(0,N)$ spaces.

The main tool we shall use is a variational inequality used in \cite{bl2,bg09,bg11,qzz13},
which was then generalized to the metric setting by Garofalo and Mondino \cite{gam14}, where the Li-Yau
type estimates were obtained on $RCD^\ast(0,N)$ spaces with $\mu(X)=1$.

In what follows, we shall let $0\le f\in L^1\cap L^\infty(X)$ and $\delta>0$ and set $f_\delta:=f+\delta$.
Moreover, for a fixed $T>0$, for each $t\in [0,T]$, we define the functional $\Phi(t)$ by
\begin{equation}\label{v-ineq}
\Phi(t):=H_{t}\lf(H_{T-t}f_\delta|\nabla \log H_{T-t}f_\delta|^2\r).
\end{equation}

\begin{lem}\label{diff-squre}
Let $(X,d,\mu)$ be  a $RCD^\ast(K,N)$ space, where $K\in \rr$ and $N\in [1,\infty)$.
Let $0\le f\in L^1(X)\cap L^\infty(X)$ and $\psi\in L^\infty(X)$. Then for each $\ez\in (0,T)$, the map
$t\mapsto \int_X |\nabla H_{T-t}f_\delta|^2\psi\,d\mu$ is absolutely continuous on $[0,T-\ez]$.
Moreover, for each $0<t<T$, it holds
\begin{eqnarray*}
\frac{\,d}{\,dt}\int_X |\nabla H_{T-t}f_\delta|^2\psi\,d\mu&&=-2\int_X \lf[\nabla H_{T-t}f_\delta\cdot \nabla \Delta H_{T-t}f_\delta\r] \psi\,d\mu
\end{eqnarray*}
\end{lem}
\begin{proof} For any $0\le s<t\le T-\ez$, it follows, from the H\"older inequality and Theorem \ref{map-hk}, that
\begin{eqnarray*}
&&\lf| \int_X |\nabla H_{T-t}f_\delta|^2\psi\,d\mu-\int_X |\nabla H_{T-s}f_\delta|^2\psi\,d\mu \r|\\
&&\quad\le  \|\psi\|_{L^\infty(X)}\int_X \lf|\nabla (H_{T-s}f_\delta+H_{T-t}f_\delta)\cdot \nabla (H_{T-s}f_\delta-H_{T-t}f_\delta)\r|\,d\mu\\
&&\quad\le  \|\psi\|_{L^\infty(X)}\lf[\int_X |\nabla (H_{T-s}f+H_{T-t}f)|^2\,d\mu\r]^{1/2}
\lf[\int_X|\nabla H_{\ez/2}(H_{T-s-\ez/2}f-H_{T-t-\ez/2}f)|^2\,d\mu\r]^{1/2}\\
&&\quad\le C \frac{\|\psi\|_{L^\infty(X)}}{\ez\wedge 1}\|f\|_{L^2(X)}\|H_{T-s-\ez/2}f-H_{T-t-\ez/2}f\|_{L^2(X)}\\
&&\quad\le C \frac{\|\psi\|_{L^\infty(X)}}{\ez\wedge 1}\|f\|_{L^2(X)}\|H_{T-s-\ez/2}f-H_{T-t-\ez/2}f\|_{L^2(X)}.
\end{eqnarray*}
Since the map $t\mapsto H_{t}f\in L^2(X)$ is absolutely continuous on $[\ez/2,T]$, the above inequality implies that the map $t\mapsto \int_X |\nabla H_{T-t}f_\delta|^2\psi\,d\mu$ is absolutely continuous on $[0,T-\ez]$. By the arbitrariness of $\ez$, we further see that the map $t\mapsto\int_X |\nabla H_{T-t}f_\delta|^2\psi\,d\mu$ is differentiable on a.e. $t\in (0,T)$.

Now, for $0<s<t<T$, it follows, from the H\"older inequality and Theorem \ref{map-hk}, that
\begin{eqnarray*}
&&\lf|\frac{1}{t-s}\lf(\int_X |\nabla H_{T-t}f_\delta|^2\psi\,d\mu-\int_X |\nabla H_{T-s}f_\delta|^2\psi\,d\mu\r)
+2\int_X \lf[\nabla H_{T-t}f_\delta\cdot \nabla \Delta H_{T-t}f_\delta\r] \psi\,d\mu\r|\\
&&\quad\le \lf|\int_X\lf[\nabla (H_{T-t}f+H_{T-s}f)\cdot \nabla H_{\frac{T-t}2} \lf(\frac{(1-H_{t-s})H_{\frac{T-t}2}f}{t-s}+\Delta H_{\frac{T-t}2}f\r)\r]\psi\,d\mu\r|\\
&&\quad\quad +\lf|\int_X \lf[\nabla (H_{T-t}f-H_{T-s}f)\cdot \nabla \Delta H_{T-t}f\r] \psi\,d\mu\r|\\
&&\quad\le C\frac{\|\psi\|_{L^\infty(X)}}{(T-t)\wedge 1} \|f\|_{L^2(X)}\lf\|\frac{(1-H_{t-s})H_{\frac{T-t}2}f}{t-s}+\Delta H_{\frac{T-t}2}f\r\|_{L^2(X)}\\
&&\quad\quad +C\frac{\|\psi\|_{L^\infty(X)}}{(T-t)\wedge 1} \|H_{t-s}f-f\|_{L^2(X)}\|\Delta H_{\frac{T-t}{2}}f\|_{L^2(X)},
\end{eqnarray*}
which tends to zero as $ s\to t$, since $\frac{(1-H_{t-s})H_{\frac{T-t}2}f}{t-s}\to -\Delta H_{\frac{T-t}2}f$
 and $H_{t-s}f\to f$ in $L^2(X)$. This implies that the required equality
holds true, and hence finishes the proof.
\end{proof}

\begin{prop}\label{v-dif}
Let $(X,d,\mu)$ be  a $RCD^\ast(K,N)$ space, where $K\in \rr$ and $N\in [1,\infty)$.
Let $0\le f,\vz\in L^1(X)\cap L^\infty(X)$,  and $T,\delta>0$.

{\rm (i)} The map $t\mapsto \int_X \Phi(t)\varphi\,d\mu$ is uniformly continuous on $[0,T-\ez]$ for each $\ez\in (0,T)$,
and is absolutely continuous on $[\ez_1,T-\ez]$ for any $\ez_1,\ez$ satisfying $0<\ez_1<T-\ez<T$.

{\rm(ii)} For a.e. $t\in [0,T]$, it holds
\begin{eqnarray*}
\frac{\,d}{\,dt}\int_X \Phi(t)\varphi\,d\mu&&=\int_X |\nabla \log H_{T-t}f_\delta|^2\Delta (H_t\vz H_{T-t}f_\delta)\,d\mu\\
&&\quad -2\int_X \nabla \log H_{T-t}f_\delta\cdot \nabla \Delta (\log H_{T-t}f_\delta) H_{T-t}f_\delta H_t\vz\,d\mu.
\end{eqnarray*}
\end{prop}
\begin{proof} (i) By the chain rule of differentials, we see that, for all $0\le \varphi\in L^1(X)\cap L^\infty(X)$ and $t\in [0,T]$,
\begin{equation}\label{w-v-ineq}
\int_X \Phi(t)\varphi\,d\mu=\int_X H_{T-t}f_\delta|\nabla \log H_{T-t}f_\delta|^2 H_t\vz\,d\mu
=\int_X \frac{|\nabla H_{T-t}f_\delta|^2}{H_{T-t}f_\delta} H_t\vz\,d\mu.
\end{equation}
Thus, for any $(s_1,s_2)\subset [0,T-\ez]$, we see that
\begin{eqnarray*}
&&\lf|\int_X \Phi(s_1)\varphi\,d\mu-\int_X \Phi(s_2)\varphi\,d\mu\r|\\
&&\quad \le\int_X \lf||\nabla H_{T-s_1}f_\delta|^2-|\nabla H_{T-s_2}f_\delta|^2\r|\frac{H_{s_1}\vz}{H_{T-s_1}f_\delta} \,d\mu+\int_X |\nabla H_{T-s_2}f_\delta|^2\lf|\frac{H_{s_1}\vz}{H_{T-s_1}f_\delta}-
\frac{H_{s_2}\vz}{H_{T-s_2}f_\delta}\r|\,d\mu\\
&&\quad =: \mathrm{I}+\mathrm{II}.
\end{eqnarray*}

The proof of Lemma \ref{diff-squre} implies that
\begin{eqnarray*}
\mathrm{I}&&\le C \frac{\|\vz\|_{L^\infty(X)}}{\delta[\ez\wedge 1]}\|f\|_{L^2(X)}\|H_{T-s_2-\ez/2}f-H_{T-s_1-\ez/2}f\|_{L^2(X)},
\end{eqnarray*}
while it follows from Theorem \ref{map-hk} and the H\"older inequality that
\begin{eqnarray*}
\mathrm{II}&&\le \int_X  |\nabla H_{T-s_2}f_\delta|^2\frac{\lf|H_{s_1}\vz H_{T-s_2}f_\delta-{H_{s_2}\vz H_{T-s_1}f_\delta}\r|}{H_{T-s_1}f_\delta H_{T-s_2}f_\delta}\,d\mu \\
&&\le   \int_X  |\nabla H_{T-s_2}f_\delta|^2\frac{\lf|H_{s_1}\vz H_{T-s_2}f-{H_{s_2}\vz H_{T-s_1}f}\r|}{H_{T-s_1}f_\delta H_{T-s_2}f_\delta}\,d\mu\\
&&\quad +\delta\int_X  |\nabla H_{T-s_2}f_\delta|^2\frac{\lf|H_{s_1}\vz-{H_{s_2}\vz }\r|}{H_{T-s_1}f_\delta H_{T-s_2}f_\delta}\,d\mu\\
&&\le C\frac{\|f\|_{L^\infty(X)}^2}{\delta^2[(T-s_2)\wedge 1]}\int_X \lf|H_{s_1}\vz H_{T-s_2}f-{H_{s_2}\vz H_{T-s_1}f}\r|\,d\mu\\
&&\quad +C\frac{\|f\|_{L^\infty(X)}}{\delta\lf(\sqrt{(T-s_2)}\wedge 1\r)}\int_X  |\nabla H_{T-s_2}f|{|H_{s_1}\vz-{H_{s_2}\vz }|}\,d\mu\\
&& \le C\frac{\|f\|_{L^\infty(X)}^2}{\delta^2[\ez\wedge 1]}\lf[\|\vz\|_{L^2(X)}\|H_{T-s_2}f-H_{T-s_1}f\|_{L^2(X)}+\|f\|_{L^2(X)}\|H_{s_2}\vz-H_{s_1}\vz\|_{L^2(X)}\r]\\
&&\quad +C\frac{\|f\|_{L^\infty(X)}\|f\|_{L^2(X)}}{\delta[\ez\wedge 1]}\lf[\|H_{s_2}\vz-H_{s_1}\vz\|_{L^2(X)}\r].
\end{eqnarray*}
Combining the estimates with the fact $H_sf\to H_tf$ in $L^2(X)$ as $s\to t$ for any $t\in [0,\infty)$, we see that
 the map $t\mapsto \int_X \Phi(t)\varphi\,d\mu$ is uniformly continuous on $[0,T-\ez]$ for any $\ez\in (0,T)$.

Moreover, since the maps $t\mapsto H_sf\in L^2(X)$, $s\mapsto H_s\vz\in L^2(X)$ are absolutely continuous on $[\ez, T]$ for arbitrarily small
$\ez>0$,
the above estimates further imply that the map $t\mapsto\int_X \Phi(t)\varphi\,d\mu$
is absolutely continuous on $[\ez_1,T-\ez]$ for any $\ez_1,\ez$ satisfying $0<\ez_1<T-\ez<T$.

(ii) From (i) we see that $\int_X \Phi(t)\varphi\,d\mu$ is differentiable on a.e. $t\in [0,T]$.
Lemma \ref{diff-squre} implies  that
\begin{eqnarray*}
\frac{\,d}{\,dt}\int_X \Phi(t)\varphi\,d\mu&&=-2\int_X \lf[\frac{\nabla H_{T-t}f_\delta\cdot \nabla \Delta H_{T-t}f_\delta}{H_{T-t}f_\delta}\r]  H_t\vz\,d\mu\\
&&\quad +\int_X |\nabla H_{T-t}f_\delta|^2\lf(\frac{\Delta H_t\vz}{ H_{T-t}f_\delta}+\frac{H_t\vz\Delta H_{T-t}f_\delta}{(H_{T-t}f_\delta)^2}\r) \,d\mu\\
&&=-2\int_X \lf[\frac{\nabla \log H_{T-t}f_\delta\cdot \nabla \Delta H_{T-t}f_\delta}{H_{T-t}f_\delta}\r] H_{T-t}f_\delta H_t\vz\,d\mu\\
&&\quad +\int_X |\nabla \log H_{T-t}f_\delta|^2\lf(\Delta H_t\vz H_{T-t}f_\delta+H_t\vz\Delta H_{T-t}f_\delta\r) \,d\mu.
\end{eqnarray*}

On the other hand, notice that it holds $\mu$-a.e. that
\begin{eqnarray*}
\nabla \log H_{T-t}f_\delta\cdot \nabla \Delta (\log H_{T-t}f_\delta)&&=\frac{\nabla \log H_{T-t}f_\delta\cdot \nabla \Delta H_{T-t}f_\delta}{H_{T-t}f_\delta}-\la \nabla \log H_{T-t}f_\delta,\nabla|\nabla \log H_{T-t}f_\delta|^2\ra\\
&&\quad -|\nabla \log H_{T-t}f_\delta|^2\frac{\Delta H_{T-t}f_\delta}{H_{T-t}f_\delta}.
\end{eqnarray*}
Thus, for a.e. $t\in (0,T)$, we have
\begin{eqnarray*}
&&\frac{\,d}{\,dt}\int_X \Phi(t)\varphi\,d\mu
\\
&&\quad=\int_X |\nabla \log H_{T-t}f_\delta|^2\lf[\Delta (H_t\vz H_{T-t}f_\delta)-2\nabla H_t\vz\cdot \nabla H_{T-t}f_\delta\r]\,d\mu\\
&&\quad\quad-2\int_X \nabla \log H_{T-t}f_\delta\cdot \nabla \Delta (\log H_{T-t}f_\delta) H_{T-t}f_\delta H_t\vz\,d\mu\\
&&\quad\quad -2\int_X \lf[\nabla \log H_{T-t}f_\delta\cdot\nabla|\nabla \log H_{T-t}f_\delta|^2+|\nabla \log H_{T-t}f_\delta|^2\frac{\Delta H_{T-t}f_\delta}{H_{T-t}f_\delta}\r]H_{T-t}f_\delta H_t\vz\,d\mu.
\end{eqnarray*}
Noticing that $|\nabla \log H_{T-t}f_\delta|^2\in W^{1,2}(X)\cap L^\infty(X)$, $H_{T-t}f,H_t\vz\in W^{1,2}(X)\cap L^\infty(X)\cap LIP(X)$, we find
\begin{eqnarray*}
&&\int_X \lf[|\nabla \log H_{T-t}f_\delta|^2\frac{\Delta H_{T-t}f_\delta}{H_{T-t}f_\delta}\r]H_{T-t}f_\delta H_t\vz\,d\mu\\
&&\quad=-\int_X |\nabla \log H_{T-t}f_\delta|^2\lf[\nabla H_t\vz\cdot \nabla H_{T-t}f_\delta\r]\,d\mu\\
&&\quad\quad-
\int_X \lf[\nabla \log H_{T-t}f_\delta\cdot\nabla|\nabla \log H_{T-t}f_\delta|^2\r]H_{T-t}f_\delta H_t\vz\,d\mu,
\end{eqnarray*}
which implies the desired estimate
\begin{eqnarray*}
\frac{\,d}{\,dt}\int_X \Phi(t)\varphi\,d\mu
&&=\int_X |\nabla \log H_{T-t}f_\delta|^2\Delta (H_t\vz H_{T-t}f_\delta)\,d\mu\\
&&\quad-2\int_X \nabla \log H_{T-t}f_\delta\cdot \nabla \Delta (\log H_{T-t}f_\delta) H_{T-t}f_\delta H_t\vz\,d\mu.
\end{eqnarray*}
The proof is then completed.
\end{proof}

\begin{prop}\label{p3.1}
Let $(X,d,\mu)$ be  a $RCD^\ast(K,N)$ space, where $K\in \rr$ and $N\in [1,\infty)$.
Assume that  $0\le f,\vz\in L^1(X)\cap L^\infty(X)$ and $T,\delta>0$. Let $a\in C^1([0,T],\rr^+)$
and $\gz\in C([0,T],\rr)$. Then, for a.e. $t\in [0,T]$, it holds that
\begin{eqnarray*}
&&\frac{\,d}{\,dt}\int_X \Phi(t)a(t)\varphi\,d\mu\\
&&\quad\ge \int_X\lf[\lf(a'(t)-\frac{4a(t)\gz(t)}{N}+2Ka(t)\r)\Phi(t)+\frac{4a(t)\gz(t)}{N}\Delta H_Tf_\delta-\frac{2a(t)\gz(t)^2}{N}H_Tf_\delta\r]\vz\,d\mu
\end{eqnarray*}
\end{prop}
\begin{proof}
By using Proposition \ref{v-dif}, one has
\begin{eqnarray*}
\frac{\,d}{\,dt}\int_X \Phi(t)a(t)\varphi\,d\mu&&=a'(t)\int_X \Phi(t)\varphi\,d\mu+a(t)\int_X |\nabla \log H_{T-t}f_\delta|^2\Delta (H_t\vz H_{T-t}f_\delta)\,d\mu\\
&&\quad -2a(t)\int_X \nabla \log H_{T-t}f_\delta\cdot \nabla \Delta (\log H_{T-t}f_\delta) H_{T-t}f_\delta H_t\vz\,d\mu.
\end{eqnarray*}

Notice that, for each $t\in (0,T)$,
$$H_{T-t}f,H_t\vz\in W^{1,2}(X)\cap L^\infty(X)\cap LIP(X)\cap \mathcal {D}(\Delta).$$
This, together with $f,\vz\ge 0$, implies that $0\le H_{T-t}f_\delta H_t\vz=(\delta+H_{T-t}f)H_t\vz\in \mathcal{D}(\Delta)\cap L^\infty(X)$,
and
$$\Delta \lf(H_{T-t}f_\delta H_t\vz\r)= H_{T-t}f\Delta H_t\vz+H_t\vz\Delta H_{T-t}f+2\nabla H_{T-t}f\cdot\nabla H_t\vz\in L^\infty(X).$$

Thus, by using Corollary \ref{bi-log} with $g=H_{T-t}f_\delta H_t\vz$, we obtain
\begin{eqnarray*}
\frac{\,d}{\,dt}\int_X \Phi(t)a(t)\varphi\,d\mu&&\ge a'(t)\int_X \Phi(t)\varphi\,d\mu+2K a(t)\int_X |\nabla \log H_{T-t}f_\delta|^2 (H_t\vz H_{T-t}f_\delta)\,d\mu\\
&&\quad +\frac{2a(t)}{N}\int_X H_{T-t}f_\delta H_t\vz(\Delta (\log H_{T-t}f_\delta))^2 \,d\mu.
\end{eqnarray*}
By using the Cauchy-Schwarz inequality, one has
$$ (\Delta (\log H_{T-t}f_\delta))^2\ge 2\gz(t)\Delta (\log H_{T-t}f_\delta)-\gz(t)^2,$$
and hence,
\begin{eqnarray*}
&&\frac{\,d}{\,dt}\int_X \Phi(t)a(t)\varphi\,d\mu\\
&&\quad\ge a'(t)\int_X \Phi(t)\varphi\,d\mu+2K a(t)\int_X |\nabla \log H_{T-t}f_\delta|^2 (H_t\vz H_{T-t}f_\delta)\,d\mu\\
&&\quad\quad +\frac{2a(t)}{N}\int_X H_{T-t}f_\delta H_t\vz\lf[2\gz(t)\Delta (\log H_{T-t}f_\delta)-\gz(t)^2\r] \,d\mu\\
&&\quad\ge \int_X\lf[\lf(a'(t)-\frac{4a(t)\gz(t)}{N}+2Ka(t)\r)\Phi(t)+\frac{4a(t)\gz(t)}{N}\Delta H_Tf_\delta-\frac{2a(t)\gz(t)^2}{N}H_Tf_\delta\r]\vz\,d\mu,
\end{eqnarray*}
as desired, and hence the proof is completed.
\end{proof}

\begin{lem}\label{limit-vineq}
Let $(X,d,\mu)$ be  a $RCD^\ast(K,N)$ space, where $K\in \rr$ and $N\in [1,\infty)$.
Let $0\le f\in L^2(X)$, $0\le\vz\in L^1(X)\cap L^\infty(X)$ and $T,\delta>0$. Let $\Phi(t)$ be as in \eqref{v-ineq}
and $a\in C^1([0,T],\rr^+)$ satisfying $a(t)=o(T-t)$ as $t\to T^-$. Then it holds
$$a(t)\int_X\Phi(t)\vz\,d\mu\to 0,\ \ \ \mbox{as}\, t\to T^-.$$
\end{lem}
\begin{proof} For each $t\in (0,T)$ close to $T$, by \eqref{w-v-ineq} and Theorem \ref{map-hk}, we obtain
\begin{eqnarray*}
\lf|(T-t)\int_X \Phi(t)\varphi\,d\mu\r|&&=(T-t)\int_X \frac{|\nabla H_{T-t}f_\delta|^2}{H_{T-t}f_\delta} H_t\vz\,d\mu\\
&&\le \|\vz\|_{L^\infty(X)}\frac{(T-t)}{\delta}\int_X |\nabla H_{T-t}f_\delta|^2\,d\mu\\
&&\le C\|\vz\|_{L^\infty(X)}\frac{\|f\|_{L^2(X)}^2}{\delta},
\end{eqnarray*}
which, together with  $a(t)=o(T-t)$, implies that
\begin{eqnarray*}
\lim_{t\to T^-}a(t)\int_X\Phi(t)\vz\,d\mu=\lim_{t\to T^-}\frac{a(t)}{T-t}\lf[(T-t)\int_X\Phi(t)\vz\,d\mu\r]=0.
\end{eqnarray*}
This finishes the proof of the lemma.
\end{proof}

We next prove a weaker version of Theorem \ref{lye}.
\begin{prop}\label{lye-weak}
Let $(X,d,\mu)$ be a $RCD^\ast(0,N)$ space with $N\in [1,\infty)$.
Assume that $u(x,t)$ is a solution to the heat equation on $X\times [0,\infty)$ with the initial value
$u(x,0)=f(x)$, where $0\le f\in L^1(X)\cap L^\infty(X)$.
Then it holds for each $T>0$ that
\begin{equation}\label{ly-restrict}
|\nabla  H_Tf|^2 -(\Delta H_Tf)(H_T f)\le \frac{N}{2T}(H_T f)^2, \ \ \mu-a.e.
\end{equation}
\end{prop}
\begin{proof}
Let $0\le \vz\in L^1(X)\cap L^\infty(X)$ be arbitrary. Following \cite{bg11,gam14}, we set
$a(t)=(1-t/T)^2$ and let $\gz$ be defined as
$$\gz(t):=\frac N4 \lf(\frac{a'(t)}{a(t)}\r).$$

By using Proposition \ref{v-dif} and  Proposition \ref{p3.1}, we deduce that, for any $\ez,\ez_1$ satisfying $0<\ez_1<T-\ez<T$,
\begin{eqnarray}\label{epsilon-ly}
&&a(T-\ez)\int_X \Phi(T-\ez)\varphi\,d\mu-a(\ez_1)\int_X \Phi(\ez_1)\vz\,d\mu\\
&&\quad=\int_{\ez_1}^{T-\ez}\frac{\,d}{\,dt}\int_X a(t)\Phi(t)\varphi\,d\mu\,dt\nonumber\\
&&\quad\ge \int_{\ez_1}^{T-\ez}\int_X\lf[\frac{4a(t)\gz(t)}{N}\Delta H_Tf_\delta-\frac{2a(t)\gz(t)^2}{N}H_Tf_\delta\r]\vz\,d\mu\,dt\nonumber\\
&&\quad\ge \lf(a(T-\ez)-a(\ez_1)\r)\int_X\Delta H_Tf_\delta\vz\,d\mu-\int_{\ez_1}^{T-\ez}\int_X\lf[\frac{2a(t)\gz(t)^2}{N}H_Tf_\delta\r]\vz\,d\mu\,dt\nonumber\\
&&\quad\ge   \lf(a(T-\ez)-a(\ez_1)\r)\int_X\Delta H_Tf_\delta \vz\,d\mu-\int_XH_Tf_\delta\vz\,d\mu \int_{\ez_1}^{T-\ez} N\frac{a'(t)^2}{8a(t)}\,dt.\nonumber
\end{eqnarray}

By Theorem \ref{map-hk}(ii) and the H\"older inequality, we have
$$\lf|\int_X \Delta H_Tf_\delta\varphi\,d\mu\r|\le \frac 1{T}\|\vz\|_{L^2(X)}\lf[\int_X |T\Delta H_Tf|^2\,d\mu\r]^{1/2} \le \frac C{T}\|\vz\|_{L^2(X)}\|f\|_{L^2(X)},$$
which together with the fact $a(t)=(1-t/T)^2$ implies
$$a(T-\ez)\int_X \Delta H_Tf_\delta\varphi\,d\mu\to 0,$$
as $\ez\to 0^+$. Based on this and Lemma \ref{limit-vineq}, by letting $\ez\to 0^+$  in \eqref{epsilon-ly}, we obtain
\begin{eqnarray*}
-a(\ez_1)\int_X \Phi(\ez_1)\vz\,d\mu&&\ge   -a(\ez_1)\int_X\Delta H_Tf_\delta \vz\,d\mu-\int_XH_Tf_\delta\vz\,d\mu \int_{\ez_1}^{T} N\frac{a'(t)^2}{8a(t)}\,dt\\
&&\ge  -a(\ez_1)\int_X\Delta H_Tf_\delta \vz\,d\mu-\frac{N}{2T}\int_XH_Tf_\delta\vz\,d\mu.\nonumber
\end{eqnarray*}
Notice that by Proposition \ref{v-dif}(i) the map $t\mapsto\int_X\Phi(t)\vz\,d\mu$ is uniformly continuous on $[0,T-\ez]$ for each $\ez\in (0,T)$.
By this, $a(t)=(1-t/T)^2$, and letting $\ez_1\to 0^+$ in the above inequality, we conclude that
\begin{eqnarray}\label{weak-lye}
-\int_XH_{T}f_\delta|\nabla \log H_{T}f_\delta|^2\vz\,d\mu= -\int_X \Phi(0)\vz\,d\mu\ge    -\int_X\Delta H_Tf_\delta \vz\,d\mu-\frac{N}{2T}\int_XH_Tf_\delta\vz\,d\mu.
\end{eqnarray}

By the arbitrariness of $\vz$ ($0\le \vz\in L^1(X)\cap L^\infty(X)$), we see that
$$|\nabla \log H_Tf_\delta|^2 H_Tf_\delta-\Delta H_Tf_\delta\le \frac{N}{2T}H_Tf_\delta, \ \ \mu-a.e.,$$
which, together with the chain rule and the fact $H_Tf_\delta>\delta$ is continuous on $X$, implies that
$$|\nabla  H_Tf|^2 -(\Delta H_Tf)(H_T f_\delta)=|\nabla \log H_Tf_\delta|^2 (H_Tf_\delta)^2 -(\Delta H_Tf_\delta)(H_T f_\delta)\le \frac{N}{2T}(H_T f_\delta)^2, \ \ \mu-a.e.$$

By letting $\dz\to 0$, we see that
\begin{equation*}
|\nabla  H_Tf|^2 -(\Delta H_Tf)(H_T f)\le \frac{N}{2T}(H_T f)^2, \ \ \mu-a.e.,
\end{equation*}
which finishes the proof.
\end{proof}

Due to Proposition \ref{lye-weak}, to prove Theorem \ref{lye}, it remains to use a density argument, which we do next.
\begin{proof}[Proof of Theorem \ref{lye}]
Suppose $0\le f\in L^q(X)$ for some $q\in [1,\infty)$.
We may choose a sequence of $f_k$ satisfying $0\le f_k\in L^1(X)\cap L^\infty(X)$ such that
$f_k\to f$ in $L^q(X)$.

Let $q\ge 2$. For each $T>0$, by Theorem \ref{map-hk}, we see that
$|\nabla  H_Tf_k|\to |\nabla H_Tf|$ in $L^q(X)$, $\Delta H_Tf_k \to \Delta H_T f$ in $L^q(X)$,
and $H_Tf_k\to H_T f$ in $L^q(X)$.
From Proposition \ref{lye-weak}, it follows that for each $k$, and for each $0\le \vz\in L^1(X)\cap L^\infty(X)$,
$$\int_X |\nabla  H_Tf_k|^2\vz\,d\mu-\int_X(\Delta H_Tf_k) H_Tf_k\vz\,d\mu\le \frac{N}{2T}\int_X(H_Tf_k)^2\vz\,d\mu. $$
By letting $k\to \infty$ and using the arbitrariness of $\vz$, we then conclude that
\begin{equation}\label{g-ly-q}
|\nabla  H_Tf|^2-(\Delta H_Tf) H_Tf\le \frac{N}{2T}(H_Tf)^2, \ \mu-a.e.
\end{equation}

Let us deal with the case $q\in [1,2)$. Fix a $x_0\in X$ and let $j\in \cn$. Using \eqref{ly-restrict},
we see that for each $0\le \vz\in L^1(X)\cap L^\infty(X)$,
\begin{equation}\label{local-ly}
\int_X |\nabla  H_Tf_k|^2\vz\chi_{B(x_0,j)}\,d\mu-\int_X(\Delta H_Tf_k) H_Tf_k\vz\chi_{B(x_0,j)}\,d\mu\le \frac{N}{2T}\int_X(H_Tf_k)^2\vz\chi_{B(x_0,j)}\,d\mu.
\end{equation}

Using the local bound Lemma \ref{local-bound} and Theorem \ref{map-hk} yields
\begin{eqnarray*}
&&\lf|\int_X |\nabla  H_Tf_k|^2\vz\chi_{B(x_0,j)}\,d\mu-\int_X |\nabla  H_Tf|^2\vz\chi_{B(x_0,j)}\,d\mu\r|\\
&&\quad\le \int_X |\nabla  H_T(f_k+f)||\nabla  H_T(f_k-f)|\vz\chi_{B(x_0,j)}\,d\mu\\
&& \quad\le \frac {C(N,K,T,j,q)}{\mu(B(x_0,j))^{1/q}} \|f+f_k\|_{L^q(X)}\int_X |\nabla  H_T(f_k-f)|\vz\chi_{B(x_0,j)}\,d\mu\\
&&\quad\le \frac {C(N,K,T,j,q)}{\mu(B(x_0,j))^{1/q}} \|f+f_k\|_{L^q(X)}\|f-f_k\|_{L^q(X)}\|\vz\|_{L^{\frac{q}{q-1}}(X)},
\end{eqnarray*}
which tends to zero as $k\to \infty$. In the same manner, we conclude that, by letting $k\to \infty$ in \eqref{local-ly}, it holds
\begin{equation*}
\int_X |\nabla  H_Tf|^2\vz\chi_{B(x_0,j)}\,d\mu-\int_X(\Delta H_Tf) H_Tf\vz\chi_{B(x_0,j)}\,d\mu\le \frac{N}{2T}\int_X(H_Tf)^2\vz\chi_{B(x_0,j)}\,d\mu.
\end{equation*}
It follows from the arbitrariness of $\vz$ that
$$|\nabla  H_Tf(x)|^2-(\Delta H_Tf(x)) H_Tf(x)\le \frac{N}{2T}(H_Tf(x))^2$$
for $\mu$-a.e. $x\in B(x_0,j)$.
Letting $j\to\infty$ yields
\begin{equation}\label{g-ly-qs}
|\nabla  H_Tf|^2 -(\Delta H_Tf)(H_T f)\le \frac{N}{2T}(H_T f)^2, \ \ \mu-a.e.
\end{equation}
This and \eqref{g-ly-q} imply the inequality \eqref{g-ly-qs} holds for each $f\in L^q(X)$, where $q\in [1,\infty)$.

From the heat kernel bounds \eqref{hk2} and the fact $H_Tf$ is continuous on $X$, one can deduce
that $H_Tf$ is locally bounded away from zero as soon as $f\neq 0$.
Hence, by using the chain rule and \eqref{g-ly-qs},
one finally deduce that, for the solution $u=H_tf$, it holds
$$|\nabla \log u(x,T)|^2 -\frac{\partial }{\,\partial t}\log u(x,T)\le \frac{N}{2T}, \ \mu-a.e. \, x\in X.$$
The proof of Theorem \ref{lye} is therefore completed.
\end{proof}

Corollary \ref{ly-heatkernel} follows immediately from Theorem \ref{lye}.

\begin{proof}[Proof of Corollary \ref{ly-heatkernel}]
For each $\ez>0$, since $0<p_\ez(\cdot,y)\in L^1(X)$, and for each $t>0$, $p_{t+\ez}(x,y)=H_t(p_\ez(\cdot,y))(x)$,
by Theorem \ref{lye}, we obtain
\begin{equation*}
|\nabla \log p_{t+\ez}|^2-\frac{\partial}{\partial t}\log p_{t+\ez}\le \frac N {2t}, \ \mu-a.e.,
\end{equation*}
which is equivalent to say
\begin{equation*}
|\nabla \log p_{t}|^2-\frac{\partial}{\partial t}\log p_{t}\le \frac N {2(t-\ez)},\ \mu-a.e.,
\end{equation*}
for each $0<\ez<t$. By the arbitrariness of $\ez$, we finally obtain
\begin{equation*}
|\nabla \log p_t|^2-\frac{\partial}{\partial t}\log p_t\le \frac N {2t}, \ \mu-a.e.,
\end{equation*}
as desired.
\end{proof}

\section{Harnack inequalities}
\hskip\parindent
By applying some methods from Garofalo-Mondino \cite{gam14} and the previous section
in proving the Li-Yau inequality, we next prove the Baudoin-Garofalo inequality (Theorem \ref{BGI})
and Harnack inequalities for the heat flow (Theorem \ref{harnack}). We would like to point out that
the proofs here are essentially from Garofalo-Mondino \cite{gam14}.

%

\begin{proof}[Proof of Theorem \ref{BGI}] For $t\in [0,T]$, let
$$a(t):=\lf[\frac{e^{-Kt/3}\lf(e^{-2Kt/3}-e^{-2KT/3}\r)}{1-e^{-2KT/3}}\r]^2,$$
and $\gz$ be defined as
$$\gz(t):=\frac N4 \lf(\frac{a'(t)}{a(t)}+2K\r).$$

Assume first $f\in L^1(X)\cap L^\infty(X)$.
Since $a(t)=o(T-t)$ as $t\to T$, similar to the proof of \eqref{weak-lye}, we see that
\begin{eqnarray}
&&-\int_X H_Tf_\delta|\nabla \log H_Tf_\delta|^2\vz\,d\mu \\
&&\quad\ge -e^{-2KT/3}\int_X\Delta H_Tf_\delta\vz\,d\mu-\frac{NK}3\frac{e^{-4KT/3}}{1-e^{-2KT/3}}\int_X H_Tf_\delta\vz\,d\mu\,dt,\nonumber
\end{eqnarray}
which, together with the arbitrariness of $\vz$, implies that
$$ H_Tf_\delta|\nabla \log H_Tf_\delta|^2\le e^{-2KT/3}\Delta H_Tf_\delta+\frac{NK}3\frac{e^{-4KT/3}}{1-e^{-2KT/3}} H_Tf_\delta\nonumber
 \ \ \, \mu-a.e.$$
By the chain rule, we see that
$$|\nabla H_Tf|^2\le e^{-2KT/3}\Delta H_Tf H_Tf_\delta+\frac{NK}3\frac{e^{-4KT/3}}{1-e^{-2KT/3}}(H_Tf_\delta)^2,\ \ \mu-a.e.$$

Letting $\delta\to 0$, we conclude that
$$|\nabla H_Tf|^2\le e^{-2KT/3}\Delta H_Tf H_Tf+\frac{NK}3\frac{e^{-4KT/3}}{1-e^{-2KT/3}}(H_Tf)^2,\ \ \mu-a.e.$$

Using a density argument similar to the proof of Theorem \ref{lye},
we obtain the desired estimate and complete the proof.
\end{proof}

We next consider the proof of the Harnack inequality for the heat flow.
Notice that, as pointed out by Garofalo and Mondino \cite{gam14},
it is not known if the quantity $|\nabla H_tf|$ is continuous on $X$, which is essential for
the arguments from \cite{ly86}. To overcome this difficulty, \cite{gam14}
worked with families of geodesics where some optimal transportation is performed, and used
 the construction of good geodesics by Rajala \cite{raj13}.

\begin{proof}[Proof of Theorem \ref{harnack}]
Using Theorem \ref{BGI}, the same proof of \cite[Theorem 1.4]{gam14} gives the desired estimates,
the details being omitted.
\end{proof}

A direct corollary is the following harnack inequality for the heat kernel.
The proof can be carried out similarly to that of Corollary \ref{ly-heatkernel}, we omit the details here.

\begin{cor}\label{harnack-heatkernel}
Let $(X,d,\mu)$ be a $RCD^\ast(K,N)$ space with $K\in \rr$ and $N\in [1,\infty)$.
Then for  all $0<s<t<\infty$ and $x,y,z\in X$, it holds that

{\rm (i)} if $K>0$,
$$p_s(x,z)\le p_t(y,z)\exp\lf\{\frac{d(x,y)^2}{4(t-s)e^{2Ks/3}}\r\}\lf(\frac{1-e^{2Kt/3}}{1-e^{2Ks/3}}\r)^{N/2};$$

{\rm (ii)} if $K=0$,
$$p_s(x,z)\le p_t(y,z)\exp\lf\{\frac{d(x,y)^2}{4(t-s)}\r\}\lf(\frac t s\r)^{N/2};$$

{\rm (iii)} if $K<0$,
$$p_s(x,z)\le p_t(y,z)\exp\lf\{\frac{d(x,y)^2}{4(t-s)e^{2Kt/3}}\r\}\lf(\frac{1-e^{2Kt/3}}{1-e^{2Ks/3}}\r)^{N/2}.$$
\end{cor}

\section{Large time behavior of heat kernels}

\hskip\parindent  In this section, we shall apply the Harnack inequality (Corollary \ref{harnack-heatkernel}) to
prove the large time behavior of heat kernels (Theorem \ref{large-time}).

\begin{thm}\label{infinity-heat}
Let $(X,d,\mu)$ be a $RCD^\ast(0,N)$ space with $N\in [1,\infty)$. Let $x_0\in X$. Then the followings are equivalent:

{\rm (i)} There exists $\theta\in (0,\infty)$ such that
$$\liminf_{R\to\infty}\frac{\mu(B(x_0,R))}{R^N}=\theta.$$

{\rm (ii)} There exists a constant $C(\theta)\in (0,\infty)$ such that, for any path $(x(t),y(t),t)\in X\times X\times (0,\infty)$ satisfying
$$d(x(t),x_0)^2+d(y(t),x_0)^2=o(t)$$
as $t\to \infty$, it holds that
$$\lim_{t\to\infty}t^{N/2}p_t(x,y)=C(\theta).$$
\end{thm}
\begin{proof} Let us first show that (i) implies (ii).
For any $0<t_1<t_2<\infty$, from Corollary \ref{harnack-heatkernel}(ii), it follows that
\begin{equation*}
t_1^{N/2}p_{t_1}(x_0,x_0)\le t_2^{N/2}p_{t_2}(x_0,x_0),
\end{equation*}
i.e., $t^{N/2}p_t(x_0,x_0)$ is an increasing function on $(0,\infty)$. On the other hand, by \eqref{hk2}, Lemma \ref{volume-growth} and the assumption, we conclude that,
for each $0<t<\infty$ and $T>0$ large enough, it holds
\begin{equation*}
t^{N/2}p_t(x_0,x_0)\le C\frac{t^{N/2}}{\mu(B(x_0,\sqrt t))}\le C\frac{t^{N/2}T^{N/2}}{t^{N/2}\mu(B(x_0,\sqrt T))}\le C\theta^{-1}.
\end{equation*}
Hence, there exists $C(\theta)>0$ such that
$$\lim_{t\to\infty}t^{N/2}p_t(x_0,x_0)=C(\theta).$$

For the general case, notice that, from the Harnack inequality (Theorem \ref{harnack}), it follows that, for any $\delta\in (0,1/2)$,
\begin{eqnarray*}
t^{N/2}p_t(x,y)&&\le [(1+\delta)t]^{N/2}p_{(1+\delta)t}(x_0,y)\exp\lf(\frac{d(x,x_0)^2}{4\delta t}\r)\\
&&\le [(1+2\delta)t]^{N/2}p_{(1+2\delta)t}(x_0,x_0)\exp\lf(\frac{d(x,x_0)^2+d(y,x_0)^2}{4\delta t}\r)
\end{eqnarray*}
and, similarly,
\begin{eqnarray*}
[(1-2\delta)t]^{N/2}p_{(1-2\delta)t}(x_0,x_0)\le t^{N/2}p_t(x,y)\exp\lf(\frac{d(x,x_0)^2+d(y,x_0)^2}{4\delta t}\r).
\end{eqnarray*}
Hence, when $(x(t),y(t),t)$ satisfies $d(x(t),x_0)^2+d(y(t),x_0)^2=o(t)$, the above two inequalities imply that
$$\lim_{t\to\infty}t^{N/2}p_t(x,y)=\lim_{t\to\infty} t^{N/2}p_t(x_0,x_0)=C(\theta).$$

Let us now show that (ii) implies (i). Notice that, from the volume growth property (Lemma \ref{volume-growth}),
the function $\frac{\mu(B(x_0,R))}{R^N}$ is a non-increasing function on $(0,\infty)$. Hence, the limit of $\frac{\mu(B(x_0,R))}{R^N}$
exists and satisfies
$$\lim_{R\to\infty}\frac{\mu(B(x_0,R))}{R^N}\ge 0.$$
Suppose that
$$\lim_{R\to\infty}\frac{\mu(B(x_0,R))}{R^N}= 0.$$
Since, from \cite{st3}, the heat kernel satisfies the lower Gaussian bounds \eqref{hk2} for each $0<t<\infty$, we obtain
\begin{equation*}
t^{N/2}p_t(x_0,x_0)\ge C\frac{t^{N/2}}{\mu(B(x_0,\sqrt t))}.
\end{equation*}
Letting $t\to\infty$, we see that
$$\lim_{t\to\infty}t^{N/2}p_t(x_0,x_0)=\infty,$$
which contradicts the assumption. Hence, there exists $\theta>0$ such that
$$\lim_{R\to\infty}\frac{\mu(B(x_0,R))}{R^N}=\theta,$$
which completes the proof.
\end{proof}

\begin{rem}\rm
Notice that, if $N\in \cn$ and $(X,d,\mu)$ is an $N$-dimensional Riemannian manifold with non-negative Ricci curvature bounds, then,
from Li \cite{li86}, it holds that
$\theta C(\theta)=\omega(N)(4\pi)^{-N/2}$, where $\omega(N)$ denotes the volume of the unit ball in $\rr^N$.
In the $RCD^\ast(0,N)$ spaces, we do not know the exact constant.
\end{rem}

\begin{proof}[Proof of Theorem \ref{large-time}]
Notice that, by Lemma \ref{volume-growth}, the assumption $\liminf_{R\to\infty}\frac{\mu(B(x_0,R))}{R^N}=\theta$ implies that
$$\lim_{R\to\infty}\frac{\mu(B(x_0,R))}{R^N}=\theta.$$
For each $t\in (0,\infty)$, let  $x(t)\equiv x$ and $y(t)\equiv y$. Notice that
$$\lim_{t\to\infty}\frac{d(x(t),x_0)^2+d(y(t),x_0)^2}{t}=\lim_{t\to\infty}\frac{d(x,x_0)^2+d(y,x_0)^2}{t}=0,$$
i.e., $d(x(t),x_0)^2+d(y(t),x_0)^2=o(t)$ as $t\to\infty$.
Hence, by applying Theorem \ref{infinity-heat}, we conclude that
$$\lim_{t\to\infty} \mu(B(x_0,\sqrt t))p_t(x,y)=\lim_{t\to\infty} \frac{\mu(B(x_0,\sqrt t))}{t^{N/2}}t^{N/2}p_t(x,y)=C(\theta)$$
with $C(\theta)\in (0,\infty)$. This finishes the proof.
\end{proof}

\subsection*{Acknowledgment}
\hskip\parindent  The author would like to thank Prof. Dachun Yang for his encouragement
and his careful reading of this manuscript. He also wishes to thank the referee for
the very detailed and valuable report.
Last but not least, he wishes to thank Nicola Gigli, Huaiqian Li,  Andrea Mondino and Huichun Zhang for useful
comments and suggestions on previous versions of the article.

\noindent Renjin Jiang

\vspace{0.1cm}
\noindent
School of Mathematical Sciences\\
Beijing Normal University\\
Laboratory of Mathematics and Complex Systems\\
Ministry of Education\\
100875, Beijing\\
People's Republic of China

\vspace{0.2cm}
\noindent{\it E-mail address}: \texttt{rejiang@bnu.edu.cn}

\end{document}